\documentclass{article}[12pt]
\usepackage[cp1250]{inputenc}

\usepackage{authblk}
\usepackage[english]{babel}
\usepackage[colorlinks,citecolor=blue,urlcolor=blue]{hyperref}          

\usepackage{a4wide}
\usepackage{multicol}
\usepackage{amsfonts}
\usepackage{graphicx}
\usepackage{caption}
\usepackage{bm}
\usepackage{bbm} 
\usepackage{amsthm,amsmath,amssymb} 
\usepackage{color}
\usepackage{picture}
\usepackage{tikz}
\usepackage{dsfont}
\usepackage{mathrsfs}
\usepackage{textcomp}
\usepackage{mathtools}

\usepackage{listings}

\let\mathbb\mathds
\newcommand{\R}{\mathbf{R}}
\newcommand{\Rd}{{\mathbf{R}^d}}
\newcommand{\N}{\mathbf{N}}
\newcommand{\vareps}{\varepsilon}
\newcommand{\even}{\mathcal{E}}
\newcommand{\odd}{\mathcal{O}}

 \newcommand{\PP}{\mathbbm{P}}
 \newcommand{\EE}{\mathbbm{E}}
 \newcommand{\HH}{\mathbbm{H}}
 \newcommand{\GG}{\mathbbm{G}}

\usepackage{color}

\definecolor{kb}{RGB}{0,150,150}

\definecolor{ts}{RGB}{255,0,0}

\definecolor{mb}{RGB}{0,0,255}

\newcommand{\set}[1]{\mathcal{#1}}
\newcommand{\one}{\mathbbm{1}}

\numberwithin{equation}{section}

\newcommand*{\lemproofname}{Proof}

\newtheoremstyle{mytheoremstyle} 
    {\topsep}                    
    {\topsep}                    
    {}                  
    {}                           
    {\scshape}                   
    {.}                          
    {.5em}                       
    {}  

\theoremstyle{mytheoremstyle}

\newtheorem{lemma}{Lemma}[section]
\newtheorem{theorem}[lemma]{Theorem}
\newtheorem{corollary}[lemma]{Corollary}
\newtheorem{definition}[lemma]{Definition}
\newtheorem{remark}[lemma]{Remark}
\newtheorem{example}[lemma]{Example}

\title{Maximum likelihood estimation for discrete exponential families and random graphs
\footnote{\textbf{Funding:} The first author was supported in part by NCN (National Science Center, Poland) grant 2018/31/G/ST1/02252 and grant 049U/0052/19 from WUST. The third author was supported in part by grant  049M/0010/19 from WUST and BGF Cotutelle scholarship.
}}
 \author[1]{Krzysztof Bogdan \thanks{corresponding author, email: Krzysztof.Bogdan@pwr.edu.pl}}
 \author[2]{ Micha\l{} Bosy \thanks{email: M.Bosy@ucl.ac.uk}}
 \author[3]{Tomasz Skalski \thanks{email: Tomasz.Skalski@pwr.edu.pl}}
 \affil[1]{
 Wroc\l{}aw University of Science and Technology, Wybrze\.{z}e Wyspia\'{n}skiego 27, 50-370 Wroc\l{}aw, Poland}
 \affil[2]{
University College London, 25 Gordon Street, WC1H 0AY London, UK}
 \affil[3]{
 Wroc\l{}aw University of Science and Technology, Wybrze\.{z}e Wyspia\'{n}skiego 27, 50-370 Wroc\l{}aw, Poland, and LAREMA, Universit\'{e} d'Angers, France}
 \date{\today}

\providecommand{\keywords}[3]{\small\textbf{Key words:} #1 \\
\indent \textbf{Mathematics Subject Classification (#2):} #3}

\begin{document}
\maketitle

\begin{abstract}
We characterize the existence of the maximum likelihood estimator for discrete exponential families.
Our criterion is simple to apply as we show in various settings,  
most notably for exponential models of random graphs.
As an application, we point out the 
size of independent identically distributed samples for which the maximum likelihood estimator exists with high probability.
\end{abstract}

{2010}{05C80, 62H12.}


\section{Introduction and preliminaries}
\label{sec:intro}

\noindent
Exponential families are of paramount importance
in probability and statistics. 
They were introduced by Fisher, Pitman, Darmois and Koopman in 1934-36 and
have many 
properties that make them indispensable in 
theory and applications, see  Lehmann and Casella \cite[Section 2.7]{MR1639875}, Barndorff-Nielsen \cite[Chapter 9]{MR489333}, Anderson \cite{MR268992}, Diaconis \cite[Chapter 9.E]{MR964069}, Diaconis and Freedman \cite{MR786142}, and Lauritzen \cite{MR754971}. 
In this paper we study \textit{discrete} exponential families,
that is 
exponential families on \textit{finite} sets. We give a  new characterization of the existence of the maximum likelihood estimator (MLE) for exponential family and the data at hand. The condition can be expressed as a linear programming problem -- we actually give several formulations. 
We also present applications, in particular for specific exponential families we give a threshold of the sample size sufficient for the existence of MLE with high probability for $i.i.d.$ samples.

 The computation of MLE is in general 
difficult with the number of variables increasing.  On the other hand, 
for given data and an exponential family, 
MLE may fail to exist.
In particular, Crain~\cite{MR0362678, MR0428559} pointed out to problems with the maximum likelihood estimation when the number of parameters is too large for the sample size. He also gave a sufficient condition for
MLE 
to exist almost surely -- the Haar condition. 

A complete characterization of the existence of MLE for rather general exponential families was
given by Barndorff-Nielsen. Namely, by \cite[Theorem 9.13]{MR489333} 
MLE for a~sample and an exponential family exists if and only if the vector of the sample means calculated for a basis of the linear space of exponents belongs to the interior of the convex hull of the pointwise range of the basis.  
This beautiful criterion is alas cumbersome to apply. 
Therefore, Jacobsen in~\cite{MR1039287} gives an alternative condition for discrete exponential families, together with applications to Cox regression, logistic regression and multiplicative Poisson models. Similar 
condition is presented by Albert and Anderson in~\cite{MR738319}
for log-linear model;
also
Haberman~\cite{MR0408098} gives a characterization of the existence of MLE for hierarchical log-linear models. His conditions can be interpreted in terms of polytope geometry, see also Eriksson et al.
\cite{MR2197157}, and  Fienberg and Rinaldo \cite{MR2985941}. 
Brown~\cite{MR882001} characterizes the existence of MLE when the log-partition function is steep and regularly convex. Additionally, he interprets the problem of finding MLE as the optimization of the Kullback-Leibler divergence.
Darroch, Lauritzen and Speed~\cite{MR568718} connect the properties of MLE in decomposable models with graph-theoretical notions, thus starting the theory of graphical models in statistics. Sufficient conditions for the existence of MLE in specific exponential families are also given by 
Stone~\cite{MR1056333} and Bogdan and Ledwina~\cite{MR1405606}.
Geyer in 
\cite{MR2685353} looks for MLE in the closure of convex exponential families. He connects the existence of MLE with the linear programming feasibility problem, and in the case of nonexistent MLE he reduces the considered exponential family until MLE exists for the family. He also applies MCMC algorithms to calculate MLE. A comparison between the conditions of Barndorff-Nielsen and Jacobsen  is discussed by Konis in~\cite{konisphdthesis}. In addition, he presents an implementation of Jacobsen's test using linear programming.
A broad survey of the history of log-linear models and further motivation for the study of the existence of MLE can be found in  Fienberg and Rinaldo~\cite{MR2363267, MR2985941}.

The main inspiration
for our work is the paper of Bogdan and Bogdan
\cite{MR1768245} characterizing the existence of MLE for exponential families of continuous functions on the unit interval. 
We propose a~similar characterization, which is new in the setting of discrete exponential families. We obtain the result by a straightforward approach, which does not depend on the delicate convex analysis of \cite{MR489333}.

The paper is composed as follows. In Section~\ref{sec:mle} we give the criterion for the existence of MLE for general discrete exponential families using the notion of the \textit{set of uniqueness}. The criterion is restated in Section~\ref{s.lp} as a linear programming problem. 
In Section~\ref{sec:zastosow} we give applications to exponential families spanned by Rademacher and Walsh functions, and to exponential families of random graphs. In particular we give sharp or plain thresholds for the sample size sufficient for the existence of MLE. Auxiliary results and direct connection to the criterion of Barndorff-Nielsen are given in Appendix~\ref{sec:app}.

\noindent
{\bf Acknowledgments:}  
We are grateful to Ma\l{}gorzata Bogdan, Piotr Cio\l{}ek, Persi Diaconis, H\'{e}l\`{e}ne Massam, Sumit Mukherjee, Krzysztof Oleszkiewicz, Krzysztof Samotij and Maciej Wilczy\'{n}ski for references, comments and discussion. The third author was partially supported by the BGF Cotutelle scholarship and would like to thank Campus France for financial support.

\subsection{
Discrete exponential family
}
\label{sec:exp_basic}

Consider
a~finite set $\set{X}\neq \emptyset$ and \emph{weight} function $\mu: \set{X}\to (0,\infty)$.
As usual, $\R^\set{X}$ is the family of all the real-valued functions on $\set{X}$.
For $\phi \in {\R^\set{X}}$ we define the \textit{partition} and the \textit{log-partition} functions,
\begin{equation}
 \label{eq:part_fun}
 Z(\phi) = \sum_{x \in \set{X}}{e^{\phi(x)} \mu(x)} ,\quad \psi(\phi) = \log Z(\phi),\qquad \qquad
\end{equation}
respectively, 
and the \textit{exponential density}
\begin{equation}\label{e.ed}p=e(\phi)= e^{\phi-\psi(\phi)}=e^\phi/Z(\phi).\end{equation} 
Clearly, $p>0$ and
$
\sum_{x \in \set{X}}p(x) \mu(x)=1.
$
For arbitrary real number $c$ we have
$\psi(\phi+c)=\psi(\phi)+c$, hence 
\begin{equation}\label{eq:ec}
e(\phi+c)=e(\phi).
\end{equation}
Moreover, for $\phi_1,\phi_2\in \R^\set{X}$ we have $e(\phi_1)=e(\phi_2)$ if and only if $\phi_1-\phi_2$ is constant.
\noindent
Consider $x_1,\ldots, x_n \in \set{X}$, a \textit{sample}. For {$\phi\in \R^\set{X}$} we denote, as usual,
$$\bar{\phi}=\frac{1}{n}\sum_{i=1}^{n}{\phi\left(x_i\right)}.$$
The \textit{likelihood function} of $p=e(\phi)$ is defined as
\begin{equation*}
L_{e(\phi)}\left(x_1,\ldots, x_n\right) = L_{p}\left(x_1,\ldots, x_n\right) = \prod_{i=1}^{n}{p(x_i)},
\end{equation*}
and 
the \textit{log-likelihood function} 
is 
\begin{equation}
\label{eq:log_like}
l_{e(\phi)}\left(x_1,\ldots, x_n\right) := \log L_{e(\phi)}\left(x_1,\ldots, x_n\right) =n \left(\bar{\phi} - \psi\left(\phi\right)\right).
\end{equation} 
Of course, {for every $c\in \R$ we have}
\begin{equation}\label{eq:ds}
l_{e(\phi+c)}(x_1,\ldots, x_n)=l_{e(\phi)}(x_1,\ldots, x_n).
\end{equation} 
We note that the 
likelihood functions {are uniformly bounded}.
Indeed, for every $\phi\in 
\R^\set{X}$,
\begin{equation}\label{eq:lbpsi}
\psi(\phi) = \log\sum_{x\in\set{X}}{e^{\phi(x)}\mu(x)} \geq \max_{\set{X}}{\phi} + \min_{\set{X}}{\log\mu},
\end{equation}
and so by \eqref{eq:log_like} and \eqref{eq:lbpsi}, 
\begin{equation}
\label{eq:log_likesb}
l_{e(\phi)}\left(x_1,\ldots, x_n\right)  \leq -n\min_{\set{X}}\log\mu
\quad \text{and} \quad
L_{e(\phi)}\left(x_1,\ldots, x_n\right) \leq (\min_{\set{X}} \mu )^{-n}.
\end{equation}

We fix a linear subspace $\set{B} \subset 
\R^\set{X}$. 
The \textit{exponential family} {spanned by $\set{B}$} is
\begin{equation}\label{e.efa}
e(\set{B}) := \{
p=e(\phi): \phi \in \set{B}\}.
\end{equation} 
Since $\set{X}$ is a finite set, $e(\set{B})$ will be called \textit{discrete exponential family}. 

We call   $\hat{p} \in e(\set{B})$ the MLE for 
$x_1,\ldots, x_n$ and $e(\set{B})$ if
\begin{align*}
L_{\hat{p}}\left(x_1,\ldots, x_n\right) &= \sup_{p \in e(\set{B})} L_{p}\left(x_1,\ldots, x_n\right), 
\end{align*}
{or, equivalently,}
\begin{align*}
l_{\hat{p}}\left(x_1,\ldots, x_n\right) &= \sup_{p \in e(\set{B})} l_{p}\left(x_1,\ldots, x_n\right).
\end{align*}

The following result is well known~(see, e.g., \cite[Theorem~2.1]{MR558392} or Diaconis~\cite[p. 177]{MR964069}), but for the reader's convenience we give a proof in Appendix~\ref{sec:proof1}.
\begin{lemma}
\label{lem:concave}
If MLE exists, then it is unique.
\end{lemma}

Despite the boundedness \eqref{eq:log_likesb}, MLE may fail to exist, as shown by the following example.
\begin{example}\label{ex:2p}
Let $\set{X} = \{0, 1\}$, $\mu\equiv 1$,
$\set{B} = \R^\set{X}$,
$n=1$ and $x_1=1$.
Let $a,b \in \R$ and $\phi = a+b\one_{\{1\}}$.
Then $Z(\phi) = e^a(1+e^b)$,
$e(\phi) = e^{b\one_{\{1\}}}/(1+e^b)$, 
and
${L}_{e(\phi)}(x_1)=e(\phi)(
{1})=e^b/(1+e^b)$. Thus,
$\sup
{L}_{e(\phi)}(x_1)=1$,
but the supremum is not attained for any $a, b\in \R$,  
so MLE
does not exist  in this case.
On the other hand, if $n=3$, $x_1=x_2=0$, and $x_3=1$, then
${L}_{e(\phi)}(x_1, x_2, x_3)=e^b/(1+e^b)^3$. By calculus, the maximum is attained when $e^b=1/2$, therefore $\hat p =(2-\one_{\{1\}})/3$ is the MLE in this case.
\end{example}
We note that the first supremum in Example~\ref{ex:2p} is approached when $b\to\infty$, that is ``at infinity'' or at the density $p=\one_{\{1\}}$, which, however, is not in $e(\R^\set{X})$ but rather in $e(\R^{\{1\}})$.
Below in Theorem~\ref{thm:mle_exist} we characterize the situation when 
the genuine MLE exists, and in Theorem~\ref{thm:mod}  we treat, by a suitable reduction {of $\set{X}$}, the case when the supremum of the likelihood function is attained ``at infinity''.
{Before we proceed, we owe the reader some comments on the notation used in this paper and in the literature.}
\subsection{Alternative settings} 
\label{ex:d}
Let $d$ be a natural number. Consider a nonempty finite set $S\subset \Rd$, weight $m$ on $S$ and the linear space spanned by the coordinate functions on $\Rd$. The corresponding exponential densities have the form
\begin{equation}\label{e.dnef}
\pi_\theta(y)=e^{\theta \cdot  y}/\zeta(\theta),\quad y\in S,
\end{equation}
where $\theta\in \Rd$, $\cdot $ is the scalar product in $\Rd$ and $\zeta(\theta) = \sum_{y \in \set{S}}{e^{\theta\cdot  y} m(y)}$.
Thus, \eqref{e.dnef} is a \textit{standard} exponential family, see \cite{MR882001}. Since the range of the vector of parameters $\theta$ is the whole of $\Rd$, which is open, the exponential family \eqref{e.dnef} is \textit{regular}, see Lauritzen \cite[Appendix~D.1]{MR1419991}.
The setting is actually generic, as we explain momentarily. If functions $\phi_1,\ldots,\phi_d$ span the linear space $\set{B}$ in the general discussion above and we let $T(x)=(\phi_1(x),\ldots,\phi_d(x))$ for $x\in \set{X}$,
then for every $\phi\in \set{B}$ there is $\theta\in \Rd$ such that $\phi(x)=\theta\cdot  T(x)$ for $x\in \set{X}$, and
\begin{equation}\label{e.efcs}
e(\phi)=e^{\theta\cdot  T}/Z(\theta\cdot  T).
\end{equation}
This is the form used by most authors, see \cite{MR1419991} or Johansen \cite{MR558392}, and $T$ is called the \textit{canonical} statistics.
Furthermore, we let $S=T(\set{X})\subset \Rd$ and $m(y)=
\sum_{x: T(x)=y}
\mu(x)$ for $y\in S$. 
With the notation of \eqref{e.dnef} and \eqref{e.efcs} we have
\begin{equation}\label{e.sefT}
\pi_\theta(y)=e(\phi)(x) \quad \text{ if } \quad T(x)=y.
\end{equation}
If $x_1,\ldots,x_n\in \set{X}$ is the sample and we denote $y_1=T(x_1),\ldots,y_n=T(x_n)$, then the corresponding likelihoods are equal, too. Therefore $\pi_{\hat \theta}$ is the maximum likelihood estimator for $y_1,\ldots,y_n$ and $\{\pi_\theta,\ \theta\in \Rd\}$
if and only if $e(\hat \theta\cdot T)$ is the maximum likelihood estimator for $x_1,\ldots,x_n$ and $\{e(\phi),\ \phi\in \set{B}\}$. This makes a complete connection between our setting and the setting of standard exponential families with finite support $S$.
We also recall that if $\phi_1,\ldots,\phi_d$ are affinely independent, then the representation \eqref{e.efcs} is  \textit{minimal}, see \cite[Chapter~1]{MR558392} or \cite{MR1419991}, where the affine independence means that $\theta\cdot T=const$
implies $\theta=0$. In general, one allows the representation to be nonminimal because over-parametrization is often natural in applications. 
We discuss this setting again in Section~\ref{s.pm} but for now
we get back to
the setting of $\set{B}$ and \eqref{e.efa}. This 
allows for using results on specific linear spaces $\set{B}$, which could be obscured by $S$ or 
$T$.

\section{Main results}
\label{sec:mle}
{Let $\one$ denotes the function on $\set{X}$ identically equal to $1$.
Assume that $\one \in \set{B}$. This entails no restriction on the considered exponential families $e(\set{B})$, but allows an elegant formulation of the criterion of existence of MLE in terms of $\set{B}$, in fact in terms of
the cone of 
non-negative functions in $\set{B}$:
\begin{equation*}
\set{B}_+ := \{\phi \in \set{B}: \phi \geq 0\}.
\end{equation*}
{We note in passing that Appendix~\ref{s.pm} gives a 
reformulation of our criterion for the existence of MLE without requiring that $\one \in \set{B}$.}

Let $U\subset \set{X}$. 
We say that $U$ is a \textit{set of uniqueness} for $\set{B}$ if $\phi=0$ is the only function in $\set{B}$ such that $\phi = 0$ on $U$. Similarly, we say that $U$ is a \textit{set of uniqueness} for $\set{B}_+$ if $\phi=0$ is the only function in $\set{B}_+$ such that $\phi = 0$ on $U$. Put differently, $U$ is of uniqueness for $\set{B}_+$ if the conditions $\phi\in \set{B}_+$ and $\phi=0$ on $U$ imply that $\phi=0$ on $\set{X}$.
{Of course, if $U$ is a set of uniqueness for $\set{B}$, then $U$ is a set of uniqueness for $\set{B}_+$.}
\begin{example}\label{ex:lf}
Let $\set{X}=\{-2,-1,0,1,2\}\subset \R$. Let $\set{B}$ denote the class of all real functions on $\set{X}$ that are of the form $a+bx$ on $\{-2,-1,0\}$ and $a+cx$ on $\{0,1,2\}$ with some $a,b,c \in\R$. Then
$\{-1, 2\}$ is a set of uniqueness for $\set{B}_+$ but $\{-2,2\}$ is not. We also observe that $\{-1, 2\}$ is not a set of uniqueness for $\set{B}$, so the non-negativity of functions in $\set{B}_+$ plays a role here.
\end{example}
Being a set of uniqueness is a monotone property in the sense that every set larger  than a set of uniqueness is also of uniqueness.
Furthermore, if $U$ is a set of uniqueness for $\set{B}_+$ and $\set{A}$ is a linear subspace of $\set{B}$, then $U$ is of uniqueness for $\set{A}_+$.

Here is a crucial definition: For $U\subset \set{X}$ and $\phi\in \set{B}$ we let
\begin{equation*}
\lambda_{U}(\phi) = \max_{\set{X}}{\phi} - \min_{U}{\phi}.
\end{equation*}

\noindent
Here is our 
characterization of the existence of MLE for discrete exponential families.
\begin{theorem}
\label{thm:mle_exist}
MLE for $e(\set{B})$ and $x_1,\ldots, x_n \in \set{X}$ exists if and only if $\{x_1,\ldots, x_n\}$ is 
of uniqueness for $\set{B}_+$.  
\end{theorem}
\begin{proof}
Let us start with the ``only if'' part. If $U=\{x_1,\ldots, x_n\}$ is not a~set of uniqueness for $\set{B}_+$, then there is a non-zero function $f \in \set{B}_+$ such that $f(x_1) = \ldots = f(x_n) = 0$. Let $\phi\in \set{B}$ be arbitrary. Let $\varphi = \phi - f$. We have $\bar{\varphi} = \bar{\phi}$, but $\psi(\varphi) <\psi(\phi)$, so by \eqref{eq:log_like}, $l_{{e(}\phi{)}}\left(x_1,\ldots, x_n\right) < l_{{e(}\varphi{)}}\left(x_1,\ldots, x_n\right)$. Therefore no $\phi \in \set{B}$ is MLE for $x_1, \ldots, x_n$.
\noindent
To prove the other implication, we let $U$ be a set of uniqueness for $\set{B}_{+}$.
By~\eqref{eq:log_like} for $\varphi \in \set{B}$,
\begin{equation*}
l_{{e(}\varphi{)}}(x_1,\ldots,x_n) = n\left(\overline{\varphi}-\psi\left(\varphi\right)\right) \leq n \left(\frac{1}{n}\left(\min_{U}{\varphi} + (n-1)\max_{\set{X}}{\varphi}\right)-\psi\left(\varphi\right)\right).
\end{equation*}
Let $C = \min_{x\in\set{X}}{\log\mu(x)}$.
By \eqref{eq:lbpsi}, \eqref{eq:ds} and Lemma~\ref{lem:comp},
\begin{align*}
l_{e(\varphi)}\left(x_1,\ldots,x_n\right) 
& \leq \min_{U}{{\varphi}} + (n-1)\max_{\set{X}}{{\varphi}} - n \max_{\set{X}}{\varphi} - nC  \\
& = -\lambda_{U}({\varphi}) - nC \rightarrow -\infty,
\end{align*}
as $\lambda_{{U}}(\varphi)\rightarrow\infty$. By Lemma~\ref{lem:comp}, if $\lambda_{U}(\varphi)\to\infty$, then $\lambda_{\set{X}}(\varphi)\rightarrow\infty$.
In particular, there exists $M > 0$ such that if
$\lambda_\set{X}(\varphi)>M$, then
\begin{equation*}
l_{e(\varphi)}(x_1,\ldots,x_n) < l_{e(0)}(x_1,\ldots,x_n) = -n\log\mu(\set{X}).
\end{equation*}
By \eqref{eq:ds} and continuity, 
the maximum of $l_{{e(}\varphi{)}}(x_1,\ldots,x_n)$ is attained on the compact set $\{\varphi\in \set{B}: 0\le \varphi\le M\}$.
The uniqueness of MLE follows from Lemma~\ref{lem:concave}.
\end{proof}
The above proof is different from that of \cite[Theorem 2.3]{MR1768245} and \cite[Theorem 9.13]{MR489333}; 
the use of $\lambda_U$ makes our arguments more direct.
\begin{remark}\label{rem:reduce}
By Theorem~\ref{thm:mle_exist} we see that the existence of MLE depends on the sequence $(x_1, \ldots, x_n)$ only through the set $\{x_1, \ldots,x_n\}$. 
Furthermore, the existence of MLE does not depend on $\mu$, i.e., we may take constant $\mu$ without loosing generality. Summarizing, the \emph{existence} of MLE depends only on $\set{B}$ and the \emph{set} $\{x_1,\ldots,x_n\}$.  The actual MLE, say $\widehat p$, depends on $\set{B}$, $\mu$, and the \emph{sequence} $(x_1,\ldots,x_n)$.
\end{remark}

\subsection{Non-existence of MLE}
\label{sec:non_mle}
In this section we elaborate on the case of nonexistence of MLE in the spirit of \cite{MR2685353}. To this end we fix $x_1,\ldots, x_n \in \set{X}$ and assume that there is a non-trivial
$\delta\in \set{B}_+$ such that $\delta(x_1)=\ldots=\delta(x_n)=0$. By Theorem~\ref{thm:mle_exist}, 
$\sup_{p \in e(\set{B})} l_{p}\left(x_1,\ldots, x_n\right)$ is not attained at any $p\in e(\set{B})$. However, the supremum is attained ``at infinity'', in fact for an exponential density on a 
subset of the state space $\set{X}$. 
Indeed, fix $\delta$ as above. If $\varphi\in \set{B}$ 
and $k\in (0,\infty)$, then
\begin{equation*} 
l_{e(\varphi)}(x_1,\ldots, x_n)\le l_{e(\varphi-k\delta)}(x_1,\ldots, x_n),
\end{equation*} 
see the first part of the proof of Theorem~\ref{thm:mle_exist}.
Furthermore,
\begin{equation}
\label{eq:psi_minus_k}
\psi\left(\varphi-k\delta\right)\to 
\log \sum_{x\in \set{X}: \delta(x)=0}{e^{\varphi(x)} \mu(x),} \quad \mbox{ as } k\to \infty.
\end{equation} 
We let 
$\widetilde{\set{X}}=\{x\in \set{X}: \delta(x)=0\}$ and carrying on with the notation for 
$\widetilde{\set{X}}$ we obtain measure $\widetilde\mu$, linear space $\widetilde{\set{B}}$ with  cone $\widetilde{\set{B}}_+$,  log-partition function $\widetilde\psi$, likelihood function $\widetilde{L}$,  log-likelihood function $\widetilde{l}$ and exponential family
$e(\widetilde{\set{B}})$. Put simpler, we discard $\{x\in \set{X}: \delta(x)>0\}$ and achieve the following reduction.
\begin{lemma}\label{lem:mod}
$\sup_{\widetilde{p} \in e(\set{\widetilde{B}})}\widetilde{l}_{\widetilde{p}}\left(x_1,\ldots, x_n\right) = \sup_{p \in e(\set{B})} l_{p}\left(x_1,\ldots, x_n\right)$.
\end{lemma}
\begin{proof}
For $\phi\in\set{B}$ we 
let $\widetilde{\phi}=\phi|_{\widetilde{\set{X}}}$. Since $\{x_1,\ldots, x_n\}\subset \widetilde{\set{X}}$,
\begin{equation}
\label{eq:mean_phi}
\overline{\widetilde{\phi}} = \frac{1}{n}\sum_{i=1}^{n}\widetilde{\phi}(x_i) = \frac{1}{n}\sum_{i=1}^{n}{\phi}(x_i) = \overline{\phi}.
\end{equation}
Furthermore,
\begin{equation*}
\psi(\phi) = \log\left(\sum_{x\in{\set{X}}}{e^{{\phi}(x)}\mu(x)}\right) \geq 
\log\left(\sum_{x\in\widetilde{\set{X}}}{e^{{\phi}\left(x\right)}{\mu}(x)}\right) =
\widetilde{\psi}(\widetilde{\phi}).
\end{equation*}
Thus $\overline{\phi}-\psi(\phi) \leq \overline{\widetilde{\phi}}-\widetilde{\psi}(\widetilde{\phi})$, and so
\begin{equation*}
\sup_{p\in e(\set{B})}l_{p}(x_1,\ldots,x_n) \leq \sup_{\widetilde{p}\in e(\widetilde{\set{B}})}\widetilde{l}_{\widetilde{p}}(x_1,\ldots,x_n).
\end{equation*}
Let $\delta\in \set{B}_+$ and $k$ be as in \eqref{eq:psi_minus_k}.
Using \eqref{eq:psi_minus_k} and \eqref{eq:mean_phi},
\begin{equation*}
l_{e(\phi-k\delta)}(x_1,\ldots,x_n)\to \widetilde{l}_{e(\widetilde{\phi})}(x_1,\ldots,x_n), \quad \mbox{ as } k\to\infty.
\end{equation*}
Therefore, 
\begin{equation*}\sup_{p\in e(\set{B})}l_{p}(x_1,\ldots,x_n) \geq \sup_{\widetilde{p}\in e(\widetilde{\set{B}})}\widetilde{l}_{\widetilde{p}}(x_1,\ldots,x_n).\end{equation*}
\end{proof}
\noindent
Motivated by Lemma~\ref{lem:mod}, we define
\begin{equation}\label{d.dB+}\{x_1,\ldots, x_n\}_{\set{B}_+}=\bigcap \phi^{-1}(\{0\}),\end{equation}
where the intersection is taken over all $\phi\in \set{B}_+$ such that $\phi(x_1)=\ldots=\phi(x_n)=0$. 
Thus for all $\phi\in \set{B}_+$, if $\phi$ vanishes on $\{x_1,\ldots, x_n\}$, then it vanishes on  
$\{x_1,\ldots, x_n\}_{\set{B}_+}$, and the latter is the largest such set.
Put differently, if there is $\delta\in \set{B}_+$ such that $\delta(x_1)=\ldots=\delta(x_n)=0$ but $\delta(x)>0$, then $x\notin \{x_1,\ldots, x_n\}_{\set{B}_+}$, and conversely.
In particular,
{$U\subset \set{X}$ is set of uniqueness for $\set{B}_+$ if and only if $U_{\set{B}_+}=\set{X}$.}
\begin{example} In the setting of Example~\ref{ex:lf}
we have $\{-2\}_{\set{B}_+}=\{-2\}$ and $\{-1\}_{\set{B}_+} = \{-2,-1,0\}$.
\end{example}

We note that if $x\not \in \{x_1,\ldots, x_n\}_{\set{B}_+}$, then there is $\phi\in {\set{B}_+}$ such that $\phi=0$ on $\{x_1,\ldots, x_n\}$ but $\phi(x)>0$.
Since $\set{X}$ is finite, by adding such functions we can construct $\delta \in {\set{B}_+}$ that vanishes precisely on $\{x_1,\ldots, x_n\}_{\set{B}_+}$, i.e., $\delta^{-1}(\{0\})=\{x_1,\ldots, x_n\}_{\set{B}_+}$. 
We adopt the setting of Lemma~\ref{lem:mod} with this $\delta$, in particular with $\widetilde{\set{X}}=\{x_1,\ldots, x_n\}_{\set{B}_+}$, and we propose the following result.
\begin{theorem}\label{thm:mod}
There is a unique $\widetilde{p}\in e(\widetilde{\set{B}})$ such that 
$\widetilde{l}_{\widetilde{p}}\left(x_1,\ldots, x_n\right) = \sup_{p \in e(\set{B})} l_{p}\left(x_1,\ldots, x_n\right)$.
\end{theorem}
\begin{proof}
\noindent 
By the definition of $\{x_1,\ldots, x_n\}_{\set{B}_+}$
and by Theorem~\ref{thm:mle_exist}, Lemma~\ref{lem:concave} and \ref{lem:mod}, there 
is a unique
$\widetilde{p}\in e(\widetilde{\set{B}})$ such that
\begin{equation*}
\widetilde{l}_{\widetilde{p}}(x_1,\ldots,x_n) = \sup_{\hat{p}\in e(\widetilde{\set{B}})}\widetilde{l}_{\hat{p}}(x_1,\ldots,x_n)=\sup_{p \in e(\set{B})} l_{p}\left(x_1,\ldots, x_n\right).
\end{equation*}
\end{proof}
\begin{example}\label{e.aaa}
For the first sample 
in Example~\ref{ex:2p} we get 
$\tilde{\set{X}}=\{x_1\}_{\set{B}_+}=\{1\}$, 
and $\tilde p=1$ on $\tilde{\set{X}}$.
\end{example}
\noindent
For more substantial applications of Theorem~\ref{thm:mod} we refer to Example~\ref{ex:nmle_allfun} and Example~\ref{e.R}.

\subsection{Linear programming}\label{s.lp}
Before we address special spaces $\set{B}$ we offer the reader a down-to-earth perspective. To start with, by a comment at the beginning of Section~\ref{sec:mle} we make the following observation.
\begin{corollary}
\label{thm:mle_exist_suf}
If $\{x_1,\ldots, x_n\}$ is of uniqueness for $\set{B}$
then MLE exists for $e(\set{B})$ and $x_1,\ldots, x_n \in \set{X}$.  
\end{corollary}
\noindent
Notably, the condition in Corollary~\ref{thm:mle_exist_suf} may be verified by solving
the following linear problem:
\begin{align*}
\phi &\in \set{B},\\
\phi(x_1)&=...=\phi(x_n)=0. 
\end{align*} 
Indeed, $\{x_1,\ldots, x_n\}$ is of uniqueness for $\set{B}$  iff the homogeneous linear system has only the trivial solution.
Now, Theorem~\ref{thm:mle_exist} is a linear programming problem. Indeed,
$\{x_1,\ldots, x_n\}$ is of uniqueness for $\set{B}_+$ iff the supremum of the (objective) function $\sum_{x\in \set{X}} \phi(x)$ is zero 
for the class of functions satisfying
\begin{align*}
\phi &\in \set{B},\\
\phi(x_1)&=...=\phi(x_n)=0,\\
\phi&\ge 0.
\end{align*}
In this vein Rinaldo, Fienberg and Zhou in~\cite[Appendix C]{MR2507456} 
observe that the condition of Barndorff-Nielsen 
is actually a linear programming problem and make connections to the geometry (of the convex hull of the set $S$ in Section~\ref{ex:d}). 
The linear programming 
also occurs 
in the study of the closures of convex exponential families \cite{MR2685353} or binary logistic regression models \cite{konisphdthesis}. 
Wang, Rauh and Massam in~\cite{MR3911110} consider the linear programming in the case when MLE fails to exist.
Since the linear programming in general runs in polynomial time, see Schrijver~\cite{MR874114}, 
it should be the method of the first choice when verifying the existence of MLE for discrete exponential families and data at hand. Having said this, for special linear spaces $\set{B}$ one can sometimes do better, as we demonstrate below.
\section{Applications}\label{sec:zastosow}
Maximization of likelihood is fundamental in estimation, model selection and testing. In many procedures it is important to know if MLE actually exists for given data $x_1,\ldots,x_n$ and  the linear space of exponents $\set{B}$; see \cite[Introduction]{MR2985941} for a list of such problems. Fienberg and Rinaldo in~\cite{MR2985941} interpret the existence of MLE by using the geometry of the polyhedral cone spanned by the rows of a specific design matrix. This result is connected with the  criterion of Barndorff-Nielsen \cite{MR489333}. 
They also inquire which parameters are estimable when MLE is missing.

Below we show that the notion of the set of uniqueness is useful in characterizing the existence of MLE in discrete exponential families for specific spaces $\set{B}$.
There are two types of results we propose:
\begin{enumerate}
 \item conditions for the existence of MLE for a given sample,
 \item probability bounds for the existence of MLE for 
independent identically distributed 
 samples.
\end{enumerate}
To this end let $\set{X}$ and $\set{B}$ be as in Section~\ref{sec:exp_basic}. 
Let  $X_1, X_2,\ldots$ be 
$i.i.d.$ 
random variables with values in $\set{X}$. 
We define the random (stopping) time
\begin{equation*}
\nu_{\text{uniq}}=\inf\{n\ge 1: \{X_1,\ldots, X_n\}\mbox{ is a set of uniqueness for }\set{B}_{+}\}.
\end{equation*}
We will estimate tails of the distribution of $\nu_{\text{uniq}}$ in terms of $\set{X}$, $\set{B}$ and $n$.
Typically we are interested in uniformly distributed $X_i$'s: $\PP(X_i=x)=1/K$, $x\in \set{X}$, $i=1,2,\ldots$, where $K=|\set{X}|$.
\subsection{All functions on $\boldsymbol{\set{X}}$}

In the setting of Theorem~\ref{thm:mle_exist} we consider
$\set{B}=\R^{\set{X}}$. We fix arbitrary $\mu>0$ on $\set{X}$, 
see Remark~\ref{rem:reduce}. Here is a trivial observation.
\begin{lemma}  
\label{lem:trivial}
MLE for $e(\R^\set{X})$ and $x_1,\ldots, x_n$ exists if and only if  $ \set{X} = \{x_1,\ldots,x_n\}$.
\end{lemma}
\begin{proof}
By Theorem~\ref{thm:mle_exist} it is enough to verify that $\set{X}$ is the only set of uniqueness for $\R^\set{X}_+$.
Obviously, $\set{X}$ is a set of uniqueness for 
$\R^\set{X}_+$ (in fact for $\R^\set{X}$).
On the other hand, if $U\subset \set{X}$ and
$x_0\in \set{X}\setminus U$, then $\one_{x_0}$ vanishes on $U$ but not on $\set{X}$, hence $U$ is not of uniqueness for $\R^\set{X}_+$ (neither it is for $\R^\set{X}$).
\end{proof}
\begin{example}\label{ex:nmle_allfun}
Using notation of Section~\ref{sec:non_mle}, we have $U_{\set{B}_+}=U$, for every $U\subset\set{X}$.
Clearly, $U\subset U_{\set{B}_+}$. On the other hand, using Equation~\eqref{d.dB+},
one may observe that for every $x \notin U$ the function $\phi(x)=\one_{\{x\}}\in\set{B}_{+}$ and $\phi=\{0\}$ on $U$, but $x\notin \phi^{-1}(\{0\})$, so $U_{\set{B}_{+}}\subset U$.
In particular, $\{x_1,\ldots,x_n\}_{\set{B}_+}=\{x_1,\ldots,x_n\}$ is the new state space $\widetilde{\set{X}}$.
\end{example}
\noindent
Later on we give examples which use the full strength of Theorem~\ref{thm:mle_exist} and the non-negativity of functions in $\set{B}_+$ therein. For now we propose a probabilistic consequence of  Lemma~\ref{lem:trivial}.
\begin{corollary}
\label{cor:coupon}
Let $\set{B}=\R^\set{X}$ and $K=|\set{X}|$.
Let $X_1, X_2,\ldots$ be independent random variables, each with uniform distribution on $\set{X}$. 
Then, for every $c\in\R$,
\begin{equation*}
\lim_{K\to\infty}\PP\left( \nu_{\text{uniq}} < K \log K + Kc \right) =  e^{-e^{-c}}.
\end{equation*}
\end{corollary}
\begin{proof}
Let 
$\nu_{\set{X}}=\inf\{n\ge 1: \{X_1,\ldots, X_n\}=\set{X}\}$. The random variable $\nu_\set{X}$ yields a connection to the classical Coupon Collector Problem, see Erd\H{o}s and R\'{e}nyi \cite{MR0150807}, and P\'{o}sfai \cite{posfai}. 
Namely, by \cite{MR0150807},
\begin{equation*}
\lim_{K\to\infty}\PP\left( \nu_{\set{X}} < K \log K + Kc \right) =  e^{-e^{-c}}.
\end{equation*}
By Lemma~\ref{lem:trivial}, $\nu_\set{X}=\nu_{\text{uniq}}$, and the proof is complete.
\end{proof}
We aim to cover with large probability the whole of $\set{X}$ by a sample of suitable size depending on $K$. 
\begin{corollary}
\label{cor:sharpt_coupon}
Let $\vareps\in (0,1)$, $K=|\set{X}|$ and $\set{B}=\R^{\set{X}}$. 
Let $X_1, X_2,\ldots$ be independent random variables, each with uniform distribution on $\set{X}$.
If $K\to\infty$, then
\begin{equation}
\label{eq:sharpt}
\PP\left(\nu_{\text{uniq}} < \left(1-\vareps\right)K\log K\right) \to 0  \quad \mbox{ and } \quad  
\PP\left(\nu_{\text{uniq}} < \left(1+\vareps\right)K\log K\right) \to 1.
\end{equation} 
\end{corollary}
\noindent
\begin{proof}
By Lemma~\ref{lem:trivial} and Corollary~\ref{cor:coupon}, for every $c\in\R$ we get
\begin{align*}
\limsup_{K\to\infty}{\PP\left(\nu_{\text{uniq}}<\left(1-\vareps\right)K\log{K} \right)} &\leq\limsup_{K\to \infty} {\PP\left( \nu_{\text{uniq}}<K\log{K}+Kc\right)} \\
& =  e^{-e^{-c}}.
\end{align*}
Thus $\lim_{K\to\infty}\PP\left(\nu_{\text{uniq}} < \left(1-\vareps\right)K\log K\right)= 0$. The second part of~\eqref{eq:sharpt} is obtained analogously.
\end{proof}
We summarize \eqref{eq:sharpt} by saying that $K\log{K}$ is a \textit{sharp threshold} of the sample size for the existence of MLE for $e(\R^\set{X})$ and uniform $i.i.d.$ samples.
Sharp thresholds are widely used in the theory of random graphs, see \cite[Equation 3]{MR0125031}. It is also convenient to use them here to indicate the minimal size of $i.i.d.$ samples that guarantees the existence of MLE with high probability.

\subsection{Rademacher functions}
\label{sec:Rademacher}

For $k \in \N$, let us consider $\set{X}=Q_k:=\{-1, 1\}^{k}$, the $k$-dimensional discrete cube with, say, the uniform  weight
$\mu(\chi) = 2^{-k}$, $\chi \in Q_k$ (but see Remark~\ref{rem:reduce}). Thus, $K=|\set{X}|=2^k$.
For $j = 1, \ldots, k$ and $\chi=(\chi_1,\ldots,\chi_k)\in Q_k$ we define the Rademacher functions:
\begin{equation*}r_j(\chi) = \chi_j,\end{equation*}
and we denote $r_0(\chi)=1$. Let \begin{equation*}\set{B}^k = \mbox{Lin}\{r_0, r_1, \ldots, r_k\}.\end{equation*} 
We define, as usual, the exponential family
\begin{equation*}
e(\set{B}^k) = \{e(r): r \in \set{B}^k\}.
\end{equation*}

\begin{theorem}
\label{thm:Rademacher}
MLE for $e(\set{B}^k)$ and $x_1,\ldots,x_n\in{Q_k}$ exists if and only if  for all $j = 1,\ldots, k$ we have $\{r_j(x_1),\ldots, r_j(x_n)\} = \{-1, 1\}$.
\end{theorem}
\begin{proof}
By Theorem~\ref{thm:mle_exist} we only need to prove that the above condition characterizes the sets of uniqueness  for $\set{B}^k_+$.
If $j \in \{1,\ldots, k\}$ is such that $r_j(x_1)=\ldots=r_j(x_n) = 1$, then we let $r = r_0 - r_j$. Clearly, $r \in \set{B}^k_+$ and $r$ is not identically zero,  
but $r(x_i) = 0$ for all $i = 1,\ldots, n$. Thus, $\{x_1,\ldots, x_n\}$ is not a~set of uniqueness for $\set{B}^k_+$. Similarly, if $r_j(x_1)=\ldots=r_j(x_n) = -1$, then we consider the function $r = r_0 + r_j \in \set{B}^k_+$. 
\noindent
For the converse implication we consider arbitrary
\begin{equation*}
r=\sum_{j=0}^{k}{a_j r_j} \in \set{B}^k_+.
\end{equation*}
Let $\chi=-(\textrm{sign}(a_1),\ldots,\textrm{sign}(a_k))$, where, say, $\textrm{sign}(0)=1$. Obviously, $\chi\in{Q_k}$, and since $r(\chi)\geq0$, we get
\begin{equation}
\label{eq:Rademacher}
a_0 \geq \sum_{j=1}^{k}{|a_j|}.
\end{equation}
Assume that $r=0$ on $\{x_1,\ldots,x_n\}$.
Let $j \in\{ 1,\ldots, k\}$. There are $x, x' \in \{x_1,\ldots, x_n\}$ such that $r_j(x) = 1$ and  $r_j(x') = -1$. We have 
\begin{equation*}
0 = r(x) + r(x') = 2 a_0 + \sum_{i\neq j}{a_i[r_i(x) + r_i(x')]}.
\end{equation*}
It follows that 
\begin{equation*}
a_0 \leq \sum_{i\neq j}|a_i|.
\end{equation*}
By \eqref{eq:Rademacher},
$a_j = 0$, for every $j\ge 1$. 
Thereby $a_0 = 0$ and $r \equiv 0$. We see that $\{x_1,\ldots, x_n\}$ is a~set of uniqueness for $\set{B}^k_+$.  
\end{proof}
\noindent
\begin{example}
Let
$x \in Q_k$ be arbitrary. By Theorem~\ref{thm:Rademacher}, MLE for $
e\left(\set{B}^k\right)$ and~$\{x, -x\}$ exists.
\end{example}
\noindent
We define the \textit{positive} and \textit{negative half-cubes}, respectively:
\begin{equation}\label{eq:defHj}
H_j^+ = \{\chi \in Q_k: r_j(\chi) =1\}, \quad H_j^- = \{\chi \in Q_k: r_j(\chi) =-1\}, \quad j = 1,\ldots, k.
\end{equation}
We note that $\set{B}^k$ is also spanned by the indicator functions of half-cubes, namely $\one_{j}^{+}=(r_0+r_j)/2$ and $\one_{j}^{-}=(r_0-r_j)/2$, $j=1, \ldots,k$.
\begin{corollary}\label{cor:rad}
MLE for $ e(\set{B}^k)$ and~$x_1,\ldots, x_n \in Q_k$ exists if and only if $\{x_1,\ldots,x_n\}$ has a nonempty intersection with each half-cube.
\end{corollary}

\begin{example}\label{e.R}
If MLE fails to exist for $e(\set{B}^k)$ and~$x_1,\ldots, x_n \in Q_k$, then the 
following analysis may shed some light on 
Theorem~\ref{thm:mod}. Let 
\begin{equation*}J=\{j \in \{1,\ldots,k\}: \{r_j(x_1),\ldots, r_j(x_n)\} = \{-1, 1\}\},\quad J'=\{1,\ldots,k\}\setminus J.\end{equation*}
Since we consider the case when MLE does not exist, by Theorem~\ref{thm:Rademacher}, $J'\neq \emptyset$.
For $j\in J'$ we let
\begin{equation*}H_j=\{\chi\in Q_k: r_j(\chi)=r_j(x_1)=\ldots=r_j(x_n)\}.\end{equation*}
Clearly, this is a half-cube, see \eqref{eq:defHj}. We will show that
\begin{equation}\label{eq:ooo}
\{x_1,\ldots,x_n\}_{\set{B}^k_+}=\bigcap_{j\in J'} H_j.
\end{equation}
We note that for $j\in J'$, $r_j$ is constant on the right-hand side of \eqref{eq:ooo}. Accordingly, the right-hand side of \eqref{eq:ooo} is isomorphic to $\{-1,1\}^{|J|}$ or to $Q_{|J|}$.

Now if $r=\sum_{j=0}^{k}{a_j r_j} \in \set{B}^k_+$ and $r(x_1)=\ldots=r(x_n)=0$, then $r=\sum_{j\in J}{a_j r_j}+c\ge 0$ on $\{-1,1\}^{|J|}$, where $c=a_0+\sum_{j\in J'}{a_j r_j(x_1)}$ is the sum of terms which are constant on $\bigcap_{j\in J'} H_j$. In the case when $J=\emptyset$, it is obvious that $\{x_1,\ldots,x_n\}_{\set{B}^k_+}=\bigcap_{j\in J'} H_j = \{x_1\}$, since $x_1 = \ldots = x_n$. However, if $J \neq \emptyset$, then by definition of $J$ and Theorem~\ref{thm:Rademacher} (with $k=|J|$ therein), $r=0$ on $\bigcap_{j\in J'} H_j$. Thus $\bigcap_{j\in J'} H_j \subset \{x_1,\ldots,x_n\}_{\set{B}^k_+}$. 
On the other hand, we observe that
 for each $j\in J'$, $\one_{H_j^c}=0$ on the sample and $\one_{H_j^c}>0$ on $H_j^c$, hence $H_j^c\cap \{x_1,\ldots,x_n\}_{\set{B}^k_+}=\emptyset$ and $\{x_1,\ldots,x_n\}_{\set{B}^k_+}\subset \bigcap_{j\in J'} H_j$.

By Theorem~\ref{thm:mod}, MLE exists for $e(\widetilde{\set{B}}^k)$ and~$x_1,\ldots, x_n$ with the measure $\widetilde{\mu}:=\mu|_{\widetilde{\set{X}}}$ on $\widetilde{\set{X}}:=\bigcap_{j\in J'} H_j$. Of course, $\widetilde{X}$ is isomorphic with $Q_{|J|}$, if we ignore the $J'$ coordinates of the points in $\widetilde{X}$. In this way we may also think that $\widetilde{\mu}$ and $x_1,\ldots, x_n$ are on $Q_{|J|}$.
Thus, one may calculate the supremum of the log-likelihood function for $e(\set{B}^k)$, $x_1,\ldots, x_n$ and $\mu$ as the maximum of a log-likelihood function on $Q_{|J|}$.
Of course, the total mass of $\widetilde{\mu}$ is a fraction of that of $\mu$.
For instance, if $\mu$ is the uniform probability weight on $Q_k$ then $\widetilde{\mu}$ is uniform with the total mass $2^{-|J'|}$,
which adds $n|J'|\log 2$ to the log-likelihood that would be obtained for $Q_{|J|}$ with the uniform probability weight, see, e.g., \eqref{e.ed}.
\end{example}
\noindent
Here is a probabilistic application of Theorem~\ref{thm:Rademacher}.
\begin{corollary}
\label{cor:prob}
Let $k \in  \N$ and $X_1, X_2, \ldots, X_n$ be independent random variables, each with uniform distribution on $Q_k$. 
Then,
\begin{align*}
\PP\left(\mbox{MLE exists for }  e(\set{B}^k) \mbox{ and } X_1, \ldots, X_n\right) &= \left(1-\frac{1}{2^{n-1}}\right)^k \\
&\geq 1 - \frac{k}{2^{n-1}} \rightarrow 1, \mbox{ as } n \rightarrow \infty.
\end{align*}
\end{corollary}
\begin{proof}
We have $\PP(X_i=x)=2^{-k}$ for all $x\in Q_k$ and $i=1,\ldots,n$.
We let $R_{ij} = r_{j}(X_{i})$ for $i=1,\ldots,n$ and $j=1,\ldots,k$. Thus, $\PP(R_{ij}=1) = \PP(R_{ij}=-1) = \frac{1}{2}$ and $\{R_{ij}\}_{i,j}$ are independent. By Theorem~\ref{thm:Rademacher},
\begin{align*}
&\PP\left(\mbox{MLE exists for }  e(\set{B}^k) \mbox{ and } X_1, \ldots, X_n\right)\\
=&\PP \big(\left\{ R_{ij}: i=1,\ldots,n \right\} = \left\{-1, 1\right\}
\text{ for } j=1,\ldots,k\big)
= \left( 1 - \frac{2}{2^{n}} \right)^k.
\end{align*}
Applying the Bernoulli inequality finishes the proof. 
\end{proof}

\begin{corollary}
\label{cor:exp_threshold}
For $k \in  \N$ let $X_1,\ldots, X_{n(k)}$ be independent random variables, each with uniform distribution on $Q_k$. 
If $n(k)=\log_{2}k+b+o(1)$ for some $b\in\R$ as $k\to \infty$, then
\begin{equation*}
\lim_{k\to\infty}{\PP\left(\mbox{MLE exists for }  e(\set{B}^k) \mbox{ and } X_1, \ldots, X_{n(k)}\right)} =  e^{-2^{1-b}}.
\end{equation*}
\end{corollary}
\begin{proof}
By Corollary~\ref{cor:prob},
\begin{align}\nonumber
{\PP\left(\mbox{MLE exists for }  e(\set{B}^k) \mbox{ and } X_1, \ldots, X_{n(k)}\right)} 
&= \left(1-\frac{1}{k\ 2^{b-1+o(1)} }\right)^{k} \\ \label{eq:limit}
&\to  e^{-2^{1-b}}, \mbox{ as } k\to\infty.
\end{align}
\end{proof}
\begin{corollary}\label{cor:sharpt}
$\log_{2}{k}$
is a sharp threshold of the sample size for the existence of MLE for $e(\set{B}^{k})$ and $i.i.d.$ uniform samples on $Q_k$.
\end{corollary}
\begin{proof}
Let $\vareps\in (0,1)$ and (the sample size) $n=n(k)<(1-\vareps) \log_2 k$. Then,
\begin{equation*}
\PP\left( \nu_{\text{uniq}}< n\right) \leq \PP\left( \nu_{\text{uniq}}< (1-\vareps)\log_{2}k \right).
\end{equation*}
For every $b\in\R$ by the equation in \eqref{eq:limit} we have 
\begin{align*}
\limsup_{k\to\infty} \PP\left( \nu_{\text{uniq}}< (1-\vareps)\log_{2}k \right) &\leq \limsup_{k\to\infty} \PP\left( \nu_{\text{uniq}}< \log_{2}k + b \right) \\
&=  e^{-2^{1-b}}.
\end{align*}
Since $b$ is arbitrary, we conclude that $\limsup_{k\to\infty}\PP\left( \nu_{\text{uniq}}< n(k)\right)=0$.
Analogously, for the sample size $n=n(k)>(1+\vareps)\log_{2}k$ we get 
\begin{equation*}
\liminf_{k\to\infty}\PP\left( \nu_{\text{uniq}}> n(k)\right)=1,
\end{equation*}
which ends the proof.
\end{proof}
The above is in stark contrast to Corollary~\ref{cor:sharpt_coupon}. Indeed, in the present setting we have $K=|Q_k|=2^k$, so the sharp threshold  for the sample size needed for the existence of MLE is $\log_{2}\log_{2} K$.
The following result on the expectation of $\nu_{\text{uniq}}$ agrees well with the sharp threshold.
\begin{lemma}
\label{lem:eisen}
Let $\nu_{\text{uniq}}$ be 
as in Corollary~\ref{cor:exp_threshold}. 
Let $H_k = \sum_{i=1}^{k}\frac{1}{k}$ be  the $k$-th harmonic number. Then,
\begin{equation*}
\frac{H_k}{\log{2}} +1 \leq \EE{(\nu_{\text{uniq}})} < \frac{H_k}{\log{2}}+2, \quad k=1,2,\ldots.
\end{equation*}
\end{lemma}
\begin{proof}
Observe that 
$\nu_{\text{uniq}} = \max\left\{ \tau_1,\ldots,\tau_k \right\}$, 
where 
\begin{equation*}
\tau_j = \min\left\{ n\ge 1:\left\{r_{j}(X_1),\ldots,r_{j}(X_n)\right\}=\left\{-1,1\right\} \right\}, \quad j=1,\ldots,k.
\end{equation*} 
From the fact that $X_1, X_2,\ldots$ are independent and uniformly distributed, we  deduce that 
\begin{equation*}
\one_{r_{j}(X_i)\neq r_{j}(X_1)},\quad i=2,3,\ldots, \quad j=1,2\ldots,
\end{equation*}
are independent with symmetric Bernoulli distribution.
Then $\tau_1,\ldots,\tau_k$ are independent, and 
\begin{equation*}
\tau_j + 1 \sim Geom\left(1/2\right)
\end{equation*}
for $j=1,\ldots,k$. 
The result follows from Eisenberg \cite{MR2382066}.
\end{proof}
In Section~\ref{sec:codim_one} we return to Rademacher functions, but for now we turn to exponential families of random graphs,
a major motivation for this work.
\section{Random graphs}  
\label{sec:graphs_not}
In this section we focus
of random graphs.  
 Their various applications can be found in Rinaldo et al. \cite{MR2507456},  Schweinberger et al. \cite{schweinberger2017exponentialfamily} and 
Mukherjee et al.
\cite{MR3798004}. What is important for us, many such models are indeed discrete exponential families.
As usual, maximum likelihood can be used to select a suitable graph model within the exponential family, see, e.g., 
Pitman \cite[Chapter 1 and 8]{MR549771} and
Bez\'{a}kov\'{a} et al
\cite{bezakova2006graph}.
In this section we characterize the existence of MLE in such context.
The theory of random graphs started with probabilistic proofs of the existence or nonexistence of specific graphs by Erd\H{o}s, see, e.g.,  Bollob\'{a}s \cite{MR1633290}. 
Asymptotic properties of~random graphs were~developed in the seminal papers of~Erd\H{o}s and~R\'enyi
\cite{MR120167, MR0125031} and 
Gilbert 
\cite{MR0108839}. Rinaldo, Fienberg and Zhou \cite{MR2507456}  discuss geometric interpretations of the existence of MLE for discrete exponential families with applications to random graphs and social networks. Chatterjee and Diaconis \cite{MR3127871} give normalizing constants that are crucial for the computation of MLE for exponential random graph models. Furthermore, they include examples when MLE fails to exist. The same authors together with Sly  discuss in~\cite{MR2857452} the asymptotic probability of the existence and uniqueness of MLE for the $\beta$-model of graphs. This allows to connect the $\beta$-model with a random uniform model of graphs with a given degree sequence, which is then explored using graphons (graph limits, see Lov\'{a}sz and Szegedy \cite{MR2274085}). They also present an algorithm for the computation of MLE in the $\beta$-model.

Perry and Wolfe \cite{perry2012null} put non-asymptotic conditions for the existence of MLE in various random graph models parameterized by vertex-specific parameters. 
Rinaldo, Petrovi\'c and Fienberg characterize the existence of MLE for $\beta$-models in~\cite{MR3113804}. They interpret the Barndorff-Nielsen's criterion using the geometry of multidimensional polytopes of vertex-degree sequences, see also \cite{MR2985941}.
Wang, Rauh and Massam~\cite{MR3911110} transfer the 
criterion into discrete hierarchical models, using the notion of simplicial complices. These models include, e.g., graphical models and Ising models. Wang, Rauh and Massam also improve the approximation of the set of estimable parameters in the case of the nonexistence of MLE, which is discussed in the setting of marginal polytopes.

 Let us start with the notation. 
Graph is~a~pair $G = \left(V, E\right)$, where $V = \{1,\ldots,N\}, N\in \N$, is the~set~of~nodes and $E$ is the~set~of~edges, i.e.,
\begin{equation*}
E \subset \tbinom{V}{2}  := \left\{\left(r,s\right): 1\leq r<s \leq N\right\}.
\end{equation*} 
We only consider simple undirected graphs (containing no loops or multiple edges).
Let $m = |E|$. 
If $m = \tbinom{N}{2}$, then the graph is~called complete and is~denoted as $K_N$. On the other hand, the empty graph (with $m = 0$) is denoted as $\overline{K_N}$.
\noindent
For graphs $G=(V,E_1)$ and $H=(V,E_2)$ we let, as usual,
\begin{align*}
G \cup H := (V,E_{1}\cup E_{2}) &, & G \cap H := (V,E_{1}\cap E_{2}).
\end{align*}
Furthermore, $G \subset H$ means that  $E_{1} \subset E_{2}$.
\noindent
Let $\set{G}_N$ be the family of all the graphs with $N$ nodes, i.e., with $V=\{1,\ldots,N\}$. 
By a \emph{random graph} we understand a~random variable $\GG$ with values in $\set{G}_N$. The families of distributions of such random variables are called~\emph{random graph models}.  
\noindent
We focus on the exponential model of random graphs $\set{G}_{N,c}$ defined as follows.

For ${1 \leq r<s \leq N}$ and $G\in \set{G}_N$,
we let
\begin{equation*}
\one_{G}(r,s) =
\begin{cases} 1, & \mbox{if } (r,s)\in{E}, \\ 
0, & \mbox{otherwise. } 
\end{cases}
\end{equation*}
We define $\chi_{r,s}:  \set{G}_N \rightarrow \{-1, 1\}$ by $\chi_{r,s}(G) = 1 - 2  \one_{G}(r,s)$.
We consider the linear space
\begin{equation*}
\set{B}^{\set{G}_N} = \mbox{Lin}\biggl\{\ 1, \chi_{r,s}(G): 1\leq r<s\leq N \biggr\}.
\end{equation*}
 Let $c\in\R^{\binom{V}{2}}$ be a corresponding vector of coefficients. Following the setting of Section~\ref{sec:exp_basic} 
we let $\mu(G)=1$ for each $G\in\set{G}_N$  (but see Remark~\ref{rem:reduce}) and consider the exponential family
\begin{equation}\label{eq:estrella}
\set{G}_{N,c}:=e(\set{B}^{\set{G}_N}) =\left\{p_c := e^{\phi_{c} - \psi(\phi_{c})}: c \in \R^{\binom{V}{2}}\right\},
\end{equation}
where 
\begin{align*}
\phi_{c}(G) = \sum_{(r,s)\in\binom{V}{2}}c_{r,s}{ \chi_{r,s}(G)}
&,&
\psi(\phi_{c}) = \mbox{log}\sum_{G\in\set{G}_N}{e^{\phi_{c}(G)}},
\end{align*}
for $G \in  \set{G}_N$, see also \eqref{eq:ec}. 
As usual,
for $p_c \in 
\set{G}_{N,c}$ we let $L_{p_c}(G_1,\ldots,G_n) = \prod_{i=1}^{n}{p_{c}}(G_i)$, etc.
\begin{lemma}
\label{lem:edge_prob}
Let $c\in\R^{\binom{V}{2}}$ and let $\GG$ be a random graph with distribution $\set{G}_{N,c}$. Let $1\leq r<s\leq N$. Then the probability of the appearance of the edge $(r,s)$ in $\GG$ equals
\begin{equation}\label{eq:ffprs}
p_{r,s} = \frac{ e^{c_{r,s}}}{1 +  e^{c_{r,s}}}.
\end{equation}
\end{lemma}
\noindent
The result is well known but for convenience a proof is given in Appendix~\ref{sec:proof9}.

\begin{lemma}
\label{lem:indep}
Let $c\in\R^{\binom{V}{2}}$ and let $\GG$ be a random graph with distribution $\set{G}_{N,c}$. Let $1\leq r_1,s_1,r_2,s_2 \leq N$, $r_1<s_1, r_2<s_2$, and $(r_1,s_1)\neq(r_2,s_2)$. Then the appearances of edges $(r_1,s_1)$ and $(r_2,s_2)$ in $\GG$ are independent events.
\end{lemma}
\noindent
The proof of the result is similar to that of Lemma~\ref{lem:edge_prob}, and can be found in Appendix~\ref{sec:proof10}.
For instance, if $p_{r,s} = p\in (0,1)$ for every edge $(r,s)$,
then the exponential random graph with distribution $\set{G}_{N,c}$ is the Erd\H{o}s-R\'{e}nyi random graph $\set{G}_{N,p}$
in \cite{MR120167,MR0125031}. The latter means that
 $\PP(e\in E(\GG))=p$ for every edge $e\in \tbinom{V}{2}$,
and the events $e\in E(\GG)$ and $f\in E(\GG)$ are  independent for different edges $e$, $f$.

\begin{theorem}~\label{thm:MLE-graphs}
MLE for $e(\set{B}^{\set{G}_N})$ and ${G_1,\ldots,G_n \in \set{G}_{N}}$ exists if and only if
\begin{align*}
\bigcup_{i=1}^{n}{G_i} = K_{N} & & \mbox{and} & & \bigcap_{i=1}^{n}{G_i} = \overline{K_N}.
\end{align*}
\end{theorem}
\begin{proof} 
By Theorem~\textbf{\ref{thm:mle_exist}},  MLE exists if and only if $\{G_1,\ldots,G_n\}$ is of uniqueness for $\set{B}^{\set{G}_N}_{+}$.

We first prove the ``only if" part of Theorem~\ref{thm:MLE-graphs}. Let us assume that there exists an edge $(r_{0},s_{0}) \notin \bigcup_{i=1}^{n}{G_i}$. Then the function $\chi_{r_{0},s_{0}} \in \set{B}^{\set{G}_N}_{+}$ equals zero on~$G_1,\ldots,G_n$, but not on~the~whole $\set{G}_{N}$.
In addition, if there is an edge $(r_{0},s_{0}) \in \bigcap_{i=1}^{n}{G_i}$, then the function $(1 + \chi_{r_{0},s_{0}}) \in \set{B}^{\set{G}_N}_{+}$ vanishes for~$G_1,\ldots,G_n$, but it is not equal to zero, e.g., for the graph $\overline{K_N}$. 

We next prove the `if' part of the theorem. Let $\phi = k_0 + \sum_{r<s}{k_{r,s} \chi_{r,s}} \in \set{B}^{\set{G}_N}_{+}$, where $k_0, k_{r,s} \in \R$ for all $1 \leq r < s \leq N$. Since $\phi(G) \geq 0$ for every $G\in\set{G}_{N}$,
\begin{equation}
\label{eq:positivity}
{k_0 \geq \sum_{r<s}|{k_{r,s}}|}.
\end{equation}
 Let $(r_0,s_0)\in\tbinom{V}{2}$. Let $\phi(G_1)=\ldots=\phi(G_n)=0$. Since $\bigcup_{i=1}^{n}{G_i} = K_{N}$ and $\bigcap_{i=1}^{n}{G_i} = \overline{K_N}$, there exists a pair of graphs $G',G'' \in \{G_1,\ldots,G_n\}$ such that $\chi_{r_0,s_0}(G')=1$, ${\chi_{r_0,s_0}(G'')=-1}$. Therefore,
\begin{align*}
&0 = \phi(G') + \phi(G'') = 2 k_0 + \sum_{r<s}{k_{r,s}\left( \chi_{r,s}(G') + \chi_{r,s}(G'') \right) }\\
&= 2 k_0 +\sum_{\substack{r<s \\ (r,s)\neq(r_0,s_0)}}{k_{r,s}\left( \chi_{r,s}(G') + \chi_{r,s}(G'') \right) }.
\end{align*}
It follows that $k_0 \leq \sum_{(r,s)\neq(r_0,s_0)}|{k_{r,s}}|$ and eventually we get $k_{r_{0},s_{0}}=0$, thanks to~\eqref{eq:positivity}. Since $(r_0,s_0)$ is arbitrary, $k_{r,s}=0$ for every ${1 \leq r<s \leq N}$. Then also $c_0 = 0$, and thus $\phi \equiv 0$. 
\end{proof}
In the above random graph model it is possible to compute explicitly the probability of the existence of MLE for $i.i.d.$ samples of graphs in $\set{G}_N$.
To this end, for ${1\leq{r}<s\leq{N}}$ we fix ${c}_{r,s} \in \R$.
By Lemma~\ref{lem:edge_prob}
the probability of the appearance of the edge $(r,s)$ in random graph $\GG$ with distribution $\set{G}_{N,c}$ is
\begin{equation*}p_{r,s} = \frac{ e^{{c_{r,s}}}}{1 +  e^{{c_{r,s}}}}.\end{equation*}
\begin{lemma}
\label{lem:graph_prob}
Let $\{\GG_1,\ldots,\GG_n\}$ be $i.i.d.$ with distribution $\set{G}_{N,c}$.
Then the probability of~the~existence of~MLE for $e(\set{B}^{\set{G}_N})$
equals
\begin{align}
\label{eq:product}
\prod_{1\leq{r}<s\leq{N}}\left(1 - p_{r,s}^{n} - \left(1-p_{r,s}\right)^{n} \right).
\end{align}
\end{lemma}
\begin{proof}
By Theorem~\textbf{\ref{thm:MLE-graphs}}, MLE for $e(\set{B}^{\set{G}_N})$ exists if and only if among the random graphs $\GG_1,\ldots,\GG_n$ every edge $(r,s)$, $1\leq r<s\leq N$, appears at least once, but not $n$ times. For every edge $(r,s)$ the above condition is satisfied with probability $1 - \left(1-p_{r,s}\right)^{n} - \left(p_{r,s}\right)^{n}$. The independence of the occurrences of different edges in $\set{G}_{N,c}$ yields the product \eqref{eq:product}.  
\end{proof}
\noindent
In particular, if $c=0$, then the probability of~the~existence of~MLE for $e(\set{B}^{\set{G}_N})$ equals \begin{equation*}\left(1-2^{1-n}\right)
^{\binom{N}{2}},\end{equation*} which is an analogue of Corollary~\ref{cor:exp_threshold}. 
From the above results we can deduce asymptotic bounds for the $i.i.d.$ sample size for which MLE exists with high probability. To this end we recall the classical result on $p=p(N)\in (0,1)$ such that $\GG$ from $\set{G}_{N,p}$ has at least one edge with high probability. 
\begin{remark}~\textrm{\cite[Lemma 1.10]{MR3675279}}
\label{rmk:edge_threshold}
Let $\GG_{N,p(N)}$ be a random graph with distribution $\set{G}_{N,p(N)}$. Then
\begin{equation*}
\lim_{N\to\infty}\PP\left(\GG_{N,p(N)}\text{ has at least one edge}\right) = \begin{cases}
0 &\mbox{if}\quad {p\left(N\right)}=o\left({N^{-2}}\right),\\
1 &\mbox{if}\quad {N^{-2}}=o\left({p\left(N\right)}\right).
\end{cases}
\end{equation*}
\end{remark}
The above may be summarized by saying that $N^{-2}$ is a \textit{threshold} for the probability $p$ such that $\GG$ with distribution $\set{G}_{N,p}$ has at least one edge. For more information on threshold functions in the theory of random graphs see e.g. Frieze and Karo\'{n}ski \cite{MR3675279}. In particular, a sharp threshold (mentioned previously) is a  threshold but the converse is not true in general.
\begin{lemma}
\label{lem:MLE_threshold}

Let $\GG_1,\ldots,\GG_n$ be $i.i.d.$ random variables with distribution $\set{G}_{N,{c}}$.
Then $\log N$ is a threshold of the sample size $n$ for
the existence of MLE for $e(\set{B}^{\set{G}_N})$.
\end{lemma}

\begin{proof}

According to Lemma~\ref{lem:graph_prob}, the probability of~the~existence of~MLE for $e(\set{B}^{\set{G}_N})$ and $\GG_1,\ldots,\GG_n$ equals \begin{equation*}
P_{\text{MLE}} = \prod_{1\leq{r}<s\leq{N}}\left(1 - p_{r,s}^{n} - \left(1-p_{r,s}\right)^{n} \right). 
\end{equation*}
We define the function 
\begin{equation}
\label{eq:real_fun}
f(x) = 1-x^w-\left(1-x\right)^w,\quad \mbox{ }x\in(0,1),\mbox{ }w\geq 2.
\end{equation}
Clearly, $f(x) = f(1-x)$ and for $w\geq{2}$ we have $f$ increasing when $0< x < \frac{1}{2}$ and decreasing when $\frac{1}{2} < x < 1$. 
Using~\eqref{eq:real_fun} we can bound $P_{\text{MLE}}$ from above by
\begin{equation*}
P_{\text{BIG}}:=
\left(1-2^{1-n}\right)^{\binom{N}{2}}.
\end{equation*}

Applying Corollary~\ref{cor:prob} and the equality in \eqref{eq:limit} for $k={\tbinom{N}{2}}$, we observe that for every $b\in\R$ and for $n=n(N) = \log_{2}{\tbinom{N}{2}} + b + o(1)$ we have $P_{\text{BIG}}\to e^{-2^{1-b}}$, as $N\to\infty$. 
Therefore, for $n(N)=o(\log N)$ we obtain $P_{\text{MLE}} \leq P_{\text{BIG}} \to 0$, as $N\to\infty.$

We consider the sample size $n=n(N)$ (depending on $N$). We will prove 
that if $\log N/n\to 0$ as $N\to \infty$, then $P_{\text{MLE}} \to 1$. To this end we bound $P_{\text{MLE}}$ from below by
\begin{equation*}
P_{\text{SMALL}}:=\left(1-p_{\text{max}}^n - \left(1-p_{\text{max}}\right)^n\right)^{\binom{N}{2}},
\end{equation*}
where $c_{\text{max}} = \max_{1\leq r<s\leq N}{|c_{r,s}|}$ and $p_{\text{max}} = 
 e^{{c_{\text{max}}}}/(1 +  e^{{c_{\text{max}}}})$.

Take $n$ independent Erd\H{o}s-R\'{e}nyi random graphs $\HH_1,\ldots,\HH_n$ with distribution $\set{G}_{N,p_{\text{max}}}$. Then the probability of the existence of MLE for $e(\set{B}^{\set{G}_N})$ and for $\HH_1,\ldots,\HH_n$ equals exactly $P_{\text{SMALL}}$. Note that intersection and union of the graphs are also Erd\H{o}s-R\'{e}nyi  random graphs, namely
\begin{align*}
\bigcap_{i=1}^{n} {\HH_i} \sim \set{G}_{N,p_{\text{max}}^n} &, & \bigcup_{i=1}^{n} {\HH_i} = \overline{\bigcap_{i=1}^{n} \overline{\HH_i}} \sim \set{G}_{N,1-q_{\text{max}}^n},
\end{align*}
where 
\begin{equation*}q_{\text{max}} := 1-p_{\text{max}} = \frac{ e^{{- c_{\text{max}}}}}{1 +  e^{{- c_{\text{max}}}}}.\end{equation*} From Remark~\ref{rmk:edge_threshold}, with high probability we have
\begin{align*}
\bigcap_{i=1}^{n} {\HH_i} = \overline{K_N} & & \mbox{and} & & \overline{\bigcup_{i=1}^{n} {\HH_i}} = \overline{K_N},
\end{align*}
provided
\begin{align*}
p_{\text{max}}^n = o(N^{-2}) & & \mbox{and} & & q_{\text{max}}^n = o(N^{-2}). 
\end{align*}
By definition, $c_{\text{max}}>0$, so $p_{\text{max}} > q_{\text{max}}$. In order to get $P_{\text{SMALL}} \to 1$ as $n\to\infty$, it suffices to have $p_{\text{max}}^n = o(N^{-2})$. If $n(N)/\log N\to\infty$ as $N\to\infty$, then the above condition is satisfied. Therefore $\log N$ is a threshold of the sample size for existence of MLE for $e(\set{B}^{\set{G}_N})$ and independent $\GG_1,\ldots,\GG_n$ from $\set{G}_{N,c}$.
\end{proof}

\section{Applications to Walsh functions}
\label{sec:codim_one}
We return to Rademacher functions to discuss the spaces spanned by their products.
Let $k \in \N$, $1 \leq q \leq k$, and
\begin{equation*}
\set{B}_q^k =  
{\mbox{Lin}\left\{w_S: S \subset \left\{1,\ldots, k\right\} \mbox{ and } |S| \leq q\right\},}
\end{equation*}
where \begin{equation*}w_{S}(x) = \prod_{i\in{S}}{r_i}(x),\quad x\in{Q_k},\quad S\subset\{1,\ldots,k\},\end{equation*} are the Walsh functions, see, e.g.,
Oleszkiewicz et al. \cite{MR3446023}.

The case $\set{B}^k_1=\set{B}^k$ was discussed in Section~\ref{sec:Rademacher} and the case $q=2$ is related to the Ising model
of ferromagnetism in statistical mechanics, see 
Wainwright and Jordan
 \cite[Example 3.1]{Ising}.
\begin{lemma}
\label{l:dim_linspace}                 
The dimension of the linear space $\set{B}_q^k$ is $ 
\sum_{j=0}^{q}\binom{k}{j}$.  
\end{lemma}
The proof of Lemma~\ref{l:dim_linspace} is given in Appendix~\ref{sec:proof_dim_linspace}.
\begin{corollary}
\label{cor:entropy}
For $q\leq\frac{k}{2}$ we have
\begin{equation*}
dim\left(\set{B}_q^k\right) \leq 2^{k H_{2}(\frac{q}{k})} \leq \left(\frac{ek}{q}\right)^q,
\end{equation*}
where $H_{2}(p) = - p\log_{2}p - (1-p)\log_{2}(1-p)$ is the binary entropy function.
\end{corollary}
The proof follows from Lemma~\ref{l:dim_linspace} and the entropy bound for the sum of binomial coefficients, see, e.g., Galvin \cite[Theorem 3.1]{galvin2014tutorial}.  

Characterization of the existence of MLE for $e(\set{B}^k_q)$ and the related sharp thresholds seem to be hard for general $q$, even for $q=2$. In the next section we discuss the products of $k-q$ Rademacher functions for fixed $q\in\N$ ($q \leq k$). We especially focus on the products of $k-1$ and $k$ Rademacher functions. 
Below we characterize the existence of MLE for $e(\set{B}_{k-1}^k)$.
As we will see, we get a qualitatively different result than that in Section~\ref{sec:Rademacher}. 
Let $\set{E}$ and $\set{O}$ be the sets of all those points in $Q_k$ that have an even and odd number of positive coordinates, respectively.
\begin{theorem}
\label{thm:codim_1}
 MLE exists for $e(\set{B}_{k-1}^k)$ and $x_1,\ldots,x_n\in Q_k$
 if and only if $\set{E}$ or ${\set{O}}\subset \{x_1,\ldots,x_n\}$.
\end{theorem}
\begin{proof}
Thanks to Theorem~\ref{thm:mle_exist}, we only need to characterize the sets of uniqueness for $\left({\set{B}_{k-1}^k}\right)_{+}$.
To this end, we consider the hypercube $G_{Q_k}$, defined as the graph with vertices in $Q_k$ and edges between all pairs of points which differ by exactly one coordinate.
Thus,
\begin{equation*}
V(G_{Q_k})=Q_k \mbox{ and } E(G_{Q_k})=\left\{ \left\{x,y\right\}\in {Q_k}{\times} {Q_k}: |\{j: r_{j}(x) \neq r_{j}(y)\}| = 1 \right\}.
\end{equation*}
Let $U=\{x_1,\ldots,x_n\}$. 
Assume that $U$ is a set of uniqueness. Let
$e \in {\set{E}}$ and $o \in {\set{O}}$. The hypercube graph $G_{Q_k}$ is connected, so there exists a path $(e,v_1,v_2,\ldots,v_{2p},o)$ in $G_{Q_k}$. Then  
\begin{align*}
&\left( \one_{\{e,v_1\}} + \one_{\{v_2,v_3\}} + \ldots + \one_{\{v_{2p},o\}} \right) 
- \left( \one_{\{v_1,v_2\}} + \one_{\{v_3,v_4\}} + \ldots + \one_{\{v_{2p-1},v_{2p}\}}\right) \\ &= \one_{\{e\}} + \one_{\{o\}}
\end{align*}
is a nontrivial nonnegative function on $Q_k$. Therefore, we must have $\{e,o\}\cap{U}\neq\emptyset$.
Then we easily conclude that $\even\subset{U}$ or $\odd\subset{U}$.

For the converse implication, we consider $q\in\{0,\ldots,k\}$ and $(k-q)$-subcubes defined by fixing $q$ coordinates:
\begin{equation}\label{eq:prz}
\bigcap_{1\leq j_1<j_2<\ldots<j_{q}\leq k} H_j,
\end{equation}
where $H_j=H_j^+$ or $H_j^-$, see \eqref{eq:defHj}. 
When $q=k-1$,
the intersection, or a $1$-cube, is a pair of points in $Q_k$ which differ by exactly one coordinate, so they have a different parity. Moreover, each such pair can be obtained in this way. Using \eqref{eq:prz}, as in the proof of Lemma~\ref{l:dim_linspace} we see that $\one_{\{e,o\}}\in \set{B}_{k-1}^k$ for each $e\in \set{E}$ and $o\in \set{O}$. Furthermore, 
each $q$-subcube of $Q_k$ with $q\ge 1$ can be covered by disjoint pairs $\{e,o\}$ as above.  
 Therefore, the functions 
$\one_{\{e,o\}}\in \set{B}_{k-1}^k$ with $e\in \set{E}$ and $o\in \set{O}$
span the linear space $\set{B}^k_{k-1}$. 

We next claim that for every $f \in \set{B}_{k-1}^k$,
\begin{equation}\label{eq:ErO}
\sum_{x\in{\set{O}}}{f(x)}=\sum_{x\in{\set{E}}}{f(x)}.
\end{equation}
Indeed, if $f=\one_{\{e,o\}}$ with $e\in \even$ and $o\in \odd$,
then the equality is true because 
both sides of \eqref{eq:ErO} are equal to $1$. Since such functions span $\set{B}_{k-1}^k$ it follows that \eqref{eq:ErO} is true for every $f \in \set{B}_{k-1}^k$.

Finally, if nonnegative $f \in \set{B}_{k-1}^k$ vanishes on $\even$, then the sum over $\odd$ also equals zero, hence $f\equiv 0$, and the same conclusion holds if we assume that $f=0$ on $\odd$. Thus $U$ is the set of uniqueness if $\odd\subset U$ or $\even\subset U$. 
\end{proof}
\noindent
We will briefly treat the case of $e(\set{B}^k_k)$, as follows.
\begin{corollary}
\label{cor:bkk}
$k2^k\log 2$ is a sharp threshold of the sample size for the existence of MLE for $e(\set{B}^{k}_{k})$ and  $i.i.d.$ samples uniform on $Q_k$.
\end{corollary}
\begin{proof}
Observe that $e(\set{B}^{k}_{k})$ is isomorphic to $e(\R^\set{X})$ for $|\set{X}| = 2^k$. The existence of MLE for $e(\set{B}^k_k)$ is characterized in (more general) Lemma~\ref{lem:trivial}, and the sharp threshold is given after Corollary~\ref{cor:sharpt_coupon}.
\end{proof}

\noindent
Corollary~\ref{cor:bkk} is in stark contrast with the result for the (smaller) space $e(\set{B}^k_1)$ because
for $e(\set{B}^k_1)$ the sharp threshold, and so the threshold, equal $\log_{2}k$, by Corollary~\ref{cor:sharpt}.
\begin{remark}
\label{rmk:SoU_Bkq}
Let $1\leq q_1 \leq q_2 \leq k$. Then every set $U$ of uniqueness for $(\set{B}_{q_2}^k)_+$ is of uniqueness for 
$(\set{B}_{q_1}^k)_+$, because 
$(\set{B}_{q_1}^k)_+ \subset (\set{B}_{q_2}^k)_+$.
\end{remark}

A characterization of the existence of MLE for $e(\set{B}^k_q)$ for arbitrary $q$, even for $q=2$, turned out to be difficult. Accordingly, we do not give a sharp threshold 
for the size of the uniform $i.i.d.$ sample needed for the existence of MLE for $e(\set{B}^{k}_{q})$.
However, the case of $e(\set{B}^{k}_{k-q})$ seems a little easier in the sense that we are able to give
the less precise threshold for the existence of MLE for $e(\set{B}^{k}_{k-q})$.
Moreover, for each fixed $q$ the threshold for $e(\set{B}^{k}_{k-q})$ is the same as for $e(\set{B}^{k}_{k})$, namely $k2^k$ as $k\to \infty$.

\begin{lemma}
\label{lem:codim_q}
Fix $q\in \N$. Then $k2^k$ is a threshold of the sample size for the existence of MLE for $e(\set{B}^{k}_{k-q})$ and $i.i.d.$ sample uniform on $Q_k$.
\end{lemma}
\begin{proof} 
If $\lim_{k\to\infty}n(k)/(k2^k) = \infty$, then by Remark~\ref{rmk:SoU_Bkq} and Corollary~\ref{cor:bkk}, for $k \to \infty$ we get
\begin{align*}
&\PP\left(\left\{X_1,\ldots,X_{n(k)}\right\} \mbox{ is of uniqueness for }\left(\set{B}^{k}_{k-q}\right)_+ \right) \\
&\geq \PP\left(\left\{X_1,\ldots,X_{n(k)}\right\}\mbox{ is of uniqueness for }\set{B}^{k}_{k} \right) \to 1,
\end{align*}
as needed.
On the other hand, every set $U$ of uniqueness for $(\set{B}_{k-q}^{k})_+$ must intersect with every subcube defined by fixing last $k-q$ coordinates, because each $q$-subcube is the support of a function in $(\set{B}_{k-q}^{k})_+$, to wit, of its indicator. 
There are $2^{k-q}$ such $q$-subcubes, each of which we can suggestively denote by $(*,\ldots,*,\vareps_{q+1},\ldots,\vareps_{k})$, where $\vareps_{q+1},\ldots,\vareps_{k}=\pm 1$.  
Observe that the family of the above subcubes is a partition of $Q_k$. We consider each $q$-subcube as a coupon in the Coupon Collector Problem. If a sample point falls into the $q$-subcube, we consider the coupon as collected. The probability of collecting a given coupon is $2^{q-k}$.
Therefore, if $n(k) = o\left(2^{k} k\right)$, hence $n(k)= o\left(2^{k-q} \left(k-q\right)\right)$, then
\begin{align*}
\PP\left(\left\{X_1,\ldots,X_{n(k)}\right\}\mbox{ is of uniqueness for }(\set{B}^{k}_{k-q})_+  \right) \to 0,   \quad \text{ as }k \to \infty,
\end{align*}
as needed.
\end{proof}

\appendix

\section{Appendix}\label{sec:app}
\label{sec:proofs}
\subsection{Proof of Lemma~\ref{lem:concave}}
\label{sec:proof1}
Let $\hat p=e(\phi_0), \widetilde p=e(\phi_1)\in e(\set{B})$ and $\widehat p\neq \widetilde p$, so that $\phi_1-\phi_0\neq const$. 
Let $\phi_t = \phi_0 + t(\phi_1-\phi_0)$, $p_{{t}}=e(\phi_t)$ for $t\in \R$ and
$l(t)=l_{p_{{t}}}(x_1,\ldots,x_n)$. We claim that $l$ is strictly concave, that is $l''<0$. 
Indeed, since $\overline{\phi_t} =\overline{\phi_0}+t\overline{\phi_1}$ is a linear function,  by \eqref{eq:log_like} we get
\begin{equation*}
l''(t) =
 - n \frac{d^2}{dt^2} \log{Z(\phi_t)}.
\end{equation*}
Let $X$ be a random variable with values in $\set{X}$ such that $\PP(X=x)=p(x)\mu(x)$. As usual, for every $f:\set{X}\to \R$
we have 
\begin{equation*}
\EE f(X)=\sum_{x\in\set{X}}f(x)p(x)\mu(x).
\end{equation*}
Clearly, $(\log Z(\phi_t))'=\tfrac{Z(\phi_t)'}{Z(\phi_t)}$ and $(\log Z(\phi_t))''=\tfrac{Z(\phi_t)''}{Z(\phi_t)}- \left( \tfrac{Z(\phi_t)'}{Z(\phi_t)}\right) ^2$. Hence, thanks to \eqref{eq:part_fun},
\begin{align*}
Z(\phi_t)'&=\sum_{x\in\set{X}}{e^{\phi_{t}(x)}\mu(x)\left(\phi_1(x)-\phi_0(x)\right)} \\
Z(\phi_t)''&=\sum_{x\in\set{X}}{e^{\phi_{t}(x)}\mu(x)\left(\phi_1(x)-\phi_0(x)\right)^{2}}.
\end{align*}
 Thus,
\begin{align*}
 \frac{Z(\phi_t)'}{Z(\phi_t)}= \EE [\phi_{1}(X) -\phi_{0}(X)] && \frac{Z(\phi_t)''}{Z(\phi_t)} = \EE[\phi_{1}(X) -\phi_{0}(X)]^2
\end{align*}
and so
\begin{equation*}
\frac{d^2}{dt^2}\log Z(\phi_t ) = \EE\left[\phi_1 (X)- \phi_0(X) - \EE (\phi_1 (X)- \phi_0(X))\right]^2> 0,
\end{equation*}
since $\phi_1 - \phi_0$ is not constant. 
Hence, $l$ is strictly concave, in particular $l(1/2)>(l(0)+l(1))/2$.
If $ \sup_{p \in e(\set{B})} {L}_{p}(x_1, \ldots,x_n)
={L}_{\widehat p}(x_1, \ldots,x_n)=
{L}_{\widetilde p}(x_1, \ldots,x_n)$, then $l(1/2) > \sup_{p \in e(\set{B})} 
l_{p}(x_1, \ldots,x_n)$, 
which is absurd; thus at most one of $\widetilde p$ and $\widehat p$ can be the MLE.
\subsection{Control by oscillations}
$\lambda_U$ defined in Section~\ref{sec:mle}
may be thought of as a specific measure of oscillation of $\phi$. 
Of course, $\lambda_U\ge 0$. Furthermore, for every $c\in\R$,
\begin{equation}\label{eq:itc}
\lambda_{U}(\phi+c)=\lambda_{U}(\phi), \quad \phi\in\set{B},  
\end{equation}
and for every (positive number) $k>0$ we have (homogeneity),
\begin{equation}\label{eq:homogeneity}
\lambda_{U}(k\phi) = k\lambda_{U}(\phi), \quad \phi \in\set{B}, k\geq{0}.
\end{equation}
If $U = \set{X}$, then $\lambda_\set{X}(-\phi)=\lambda_\set{X}(\phi)$ for $\phi \in \set{B}$, and so $\lambda_\set{X}$ is a seminorm. Clearly, $\lambda_{U}\leq \lambda_\set{X}$. However,
if there is a nontrivial $\phi\in \set{B}_+$ such that $\phi=0$ on $U$, then $\lambda_U(\phi)=\sup_\set{X} \phi>0$ but $\lambda_U(-\phi)=0$. {The following result is the  engine of Theorem~\ref{thm:mle_exist}.}
\begin{lemma}
\label{lem:comp}
$U\subset \set{X}$ is the set of uniqueness for $\set{B}_+$ if and only if $\lambda_{U}$ is comparable with $\lambda_\set{X}$ on $\set{B}$, i.e., there exist constants $c_1, c_2>0$  such that $c_1 \lambda_\set{X}(\phi) \leq \lambda_U(\phi) \leq \lambda_\set{X}(\phi)$ for all $\phi \in \set{B}$.
\end{lemma}
\begin{proof}
We first prove the ``if'' part. Assume $U$ is not a set of uniqueness for $\set{B}_+$. 
Then there exists a nonzero function $\phi \in \set{B}_+$ such that $\phi=0$ on $U$.
We have $\lambda_{U}(-\phi) = 0$ and $\lambda_{\set{X}}(-\phi) > 0$, hence $\lambda_U$ and $\lambda_{\set{X}}$ are not comparable on $\set{B}$.  

We now prove the ``only if'' part, which is delicate.
For all $\vartheta, \phi \in \set{B}$ we have
\begin{align*}
\lambda_{U}(\vartheta + \phi)&
\leq {\max_{\set{X}} {\vartheta}} + {\max_{\set{X}}\phi} - {\min_{U} {\vartheta}} - {\min_{U} {\phi}}
 \\
& =\lambda_{U}(\vartheta) + \lambda_{U}(\phi) \leq 
\lambda_{U}(\vartheta) + \lambda_\set{X}(\phi).
\end{align*}
It follows that $\lambda_{U}(\vartheta) \geq {\lambda_{U}}(\vartheta-\phi) - \lambda_\set{X}(\phi)$, hence
\begin{equation*}
\lambda_{U}(\vartheta + \phi) \geq {\lambda_{U}}(\vartheta) - \lambda_\set{X}(\phi).
\end{equation*}
Therefore, $\vert{\lambda_{U}(\vartheta + \phi) - \lambda_{U}(\vartheta)\vert} \leq \lambda_\set{X}(\phi)$. 
As a consequence, $\lambda_U$ is continuous on $\set{B}$.

We will prove that there is a number $h>0$ such that $\lambda_U(\phi) \geq h \lambda_\set{X}(\phi)$ for every $\phi\in \set{B}$.  Let $\set{S}=\{\phi\in \set{B}: \min_\set{X} {\phi}=0 \text{ and } \max_\set{X} {\phi}=1\}$.
Let $\phi\in \set{S}$.
If $\lambda_U(\phi)=0$, then $\phi=1$ on $U$. Consider $\varphi=1-\phi$. Clearly, $\varphi\ge 0$ and $\varphi=0$ on $U$. It follows that $\varphi=0$ on $\set{X}$, because $U$ is of uniqueness. Then $\phi\equiv 1$, which contradicts the assumption $\phi\in \set{S}$.
Therefore, $\lambda_U(\phi)>0$.
Since $\set{S}$ is compact and $\lambda_{U}$ is continuous, $h:=\min_\set{S} \lambda_U>0$.
By \eqref{eq:homogeneity} and \eqref{eq:itc} we obtain $\lambda_{U}(\phi) \geq h\lambda_\set{X}(\phi)$ for all $\phi \in \set{B}$.
\end{proof}

\subsection{Proof of Lemma~\ref{lem:edge_prob}}
\label{sec:proof9}
By \eqref{eq:estrella}, each $G\in\set{G}_N$ appears in $\set{G}_{N,c}$ with probability $p_{c}(G) =  e^{\phi_{c}(G) - \psi(\phi_c)}$. Then,
\begin{align}
p_{r,s}& = \PP\left(\left(r,s\right)\in{E\left(\GG\right)} \right) = \sum_{\substack{G\in\set{G}_N\\(r,s)\in{E(G)}}}\frac{ e^{\phi_{c}(G)}}{\sum_{G\in\set{G}_N} e^{\phi_{c}(G)}}\nonumber\\
&\nonumber\\
&=\frac{\sum_{\substack{G\in\set{G}_N\\(r,s)\in{E(G)}}} e^{\phi_{c}(G)}}{\sum_{\substack{G\in\set{G}_N\\(r,s)\in{E(G)}}} e^{\phi_{c}(G)} + \sum_{\substack{G\in\set{G}_N\\(r,s)\notin{E(G)}}} e^{\phi_{c}(G)}}\nonumber\\
\label{eq:finaljumbo}
&= \frac{\sum_{\substack{G\in\set{G}_N\\(r,s)\in{E(G)}}}
 e^{
\sum_{(k,l)\in\binom{V}{2}}{c_{k,l}\chi_{k,l}(G)}
}
}{\sum_{\substack{G\in\set{G}_N\\(r,s)\in{E(G)}}} e^{\sum_{(k,l)\in\binom{V}{2}}{c_{k,l}\chi_{k,l}(G)}} + \sum_{\substack{{G\in\set{G}_N}\\{(r,s)\notin{E(G)}}}} e^{\sum_{(k,l)\in\binom{V}{2}}{c_{k,l}\chi_{k,l}(G)}}}.
\end{align}
Note that 
\begin{equation*}
\sum_{(k,l)\in\binom{V}{2}}{c_{k,l}\chi_{k,l}(G)} = c_{r,s}\chi_{r,s}(G) + C(G), 
\end{equation*}
where \begin{equation*}C(G)=\sum_{\substack{(k,l)\in\binom{V}{2}\\(k,l)\neq(r,s)}}{c_{k,l}\chi_{k,l}(G)}.\end{equation*}
Therefore 
\begin{equation*}
 e^{\sum_{(k,l)\in\binom{V}{2}}{c_{k,l}\chi_{k,l}(G)} } =  e^{c_{r,s}\chi_{r,s}(G)} \   e^{C(G)}.
\end{equation*}
Clearly, $c_{r,s}\chi_{r,s}(G)$ is $c_{r,s}$ if $(r,s)\in{E(G)}$ and it is $0$ if $(r,s)\notin{E(G)}$. Thus, \eqref{eq:finaljumbo} equals
\begin{equation*}
\frac{ e^{c_{r,s}} \  \sum_{\substack{G\in\set{G}_N\\(r,s)\in{E(G)}}}C(G)}{\sum_{\substack{G\in\set{G}_N\\(r,s)\in{E(G)}}} e^{C(G)} +  e^{c_{r,s}} \  \sum_{\substack{{G\in\set{G}_N}\\{(r,s)\notin{E(G)}}}} e^{ C(G)}}.
\end{equation*}
Let $S$ be
the graph with only one edge $(r,s)$. The map $G \mapsto G\setminus S$ is a bijection between the graphs with the edge $(r,s)$ and graphs without $(r,s)$. In addition, $C(G)=C(G\setminus S)$, and so we get
\eqref{eq:ffprs}.
\subsection{Proof of Lemma~\ref{lem:indep}}
\label{sec:proof10}
By
\eqref{eq:estrella},
each $G\in\set{G}_N$ appears in $\set{G}_{N,c}$ with probability $p_{c}(G) =  e^{\phi_{c}(G) - \psi(\phi_c)}$. Then,
\begin{equation*}
\PP\left(\left(r_1,s_1\right),\left(r_2,s_2\right)\in{E\left(\GG\right)} \right) = \sum_{\substack{G\in\set{G}_N\\(r_1,s_1),(r_2,s_2)\in{E(G)}}}\frac{ e^{\phi_{c}(G)}}{\sum_{G\in\set{G}_N} e^{\phi_{c}(G)}}.
\end{equation*}
As in the proof of Lemma~\ref{lem:edge_prob}, we observe that
\begin{equation*}
\sum_{(k,l)\in\binom{V}{2}}{c_{k,l}\chi_{k,l}(G)} = c_{r_1,s_1}\chi_{r_1,s_1}(G) + c_{r_2,s_2}\chi_{r_2,s_2}(G) +\widetilde C(G), 
\end{equation*}
where \begin{equation*}\widetilde C(G)=\sum_{\substack{(k,l)\in\binom{V}{2}\\(k,l)\neq(r_1,s_1)\\(k,l)\neq(r_2,s_2)}}{c_{k,l}\chi_{k,l}(G)}.\end{equation*}
Thus, 
\begin{equation*}
 e^{\sum_{(k,l)\in\binom{V}{2}}{c_{k,l}\chi_{k,l}(G)} } =  e^{c_{r_1,s_1}\chi_{r_1,s_1}(G)} \  e^{c_{r_2,s_2}\chi_{r_2,s_2}(G)} \  e^{\widetilde{C}(G)}.
\end{equation*}
Let $S_1$ and 
$S_2$
be 
the graphs with only one edge, $(r_1,s_1)$ and $(r_2,s_2)$, respectively. 
Let 
\begin{align*}
&\set{G}_{N_{12}}=\left\{G\in\set{G}_N: S_1 \subset G, S_2 \subset G \right\}, \\
&\set{G}_{N_{10}}=\left\{G\in\set{G}_N: S_1 \subset G, S_2 \not\subset G \right\}, \\
&\set{G}_{N_{02}}=\left\{G\in\set{G}_N: S_1 \not\subset G, S_2 \subset G \right\}, \\
&\set{G}_{N_{00}}=\left\{G\in\set{G}_N: S_1 \not\subset G, S_2 \not\subset G \right\}.
\end{align*}
a partition of $\set{G}_N$.
We observe that the maps 
\begin{equation*}
G \mapsto G\setminus S_1, \quad G \mapsto G\setminus S_2, \quad G \mapsto G\setminus (S_1\cup S_2)
\end{equation*} 
are bijections between $\set{G}_{N_{10}}$, $\set{G}_{N_{02}}$, $\set{G}_{N_{12}}$,  respectively,  and $\set{G}_{N_{00}}$. Also, for every $G\in\set{G}_N$,
\begin{equation*}
\widetilde{C}(G) = \widetilde{C}(G\setminus S_1) = \widetilde{C}(G\setminus S_2) = \widetilde{C}(G\setminus{(S_1\cup S_2)}).
\end{equation*}
Put differently, $\widetilde C(G)$ does not depend on the edges $(r_1,s_1)$ and $(r_2,s_2)$.
As in the proof of Lemma~\ref{lem:edge_prob}, we obtain
\begin{align*}
&\PP\left(\left(r_1,s_1\right),\left(r_2,s_2\right)\in{E\left(\GG\right)} \right)\\
&=\frac{ e^{c_{r_1,s_1}}   e^{c_{r_2,s_2}}}{1+  e^{c_{r_1,s_1}}  +  e^{c_{r_2,s_2}}   +  e^{c_{r_1,s_1}}   e^{c_{r_2,s_2}}} 
= p_{r_1,s_1} \ p_{r_2,s_2}.
\end{align*}

\subsection{Proof of Lemma~\ref{l:dim_linspace}}
\label{sec:proof_dim_linspace}
\begin{proof}
Consider the positive half-cubes
$H^+_1, \ldots, H^+_k$. Let 
\begin{equation*}
{\set{B}} =  
{\mbox{Lin}\left\{\prod_{i \in I_q} \one_{H_{i}^{+}}: I_q \subset \left\{0,\ldots, k\right\} \mbox{ and } |I_q| \leq q\right\}.}
\end{equation*}
We have 
${\set{B}} = \set{B}_q^k$,
because $r_0=\one_{Q_k}$,  {$r_{i}=2\one_{H_{i}^{+}}-\one_{Q_k}$} and by induction it is easy to see that for every $S \subset \left\{1,\ldots, k\right\} \mbox{ and } |S| < q$, if Walsh function $w_{S} \in {\set{B}}$ then their product with Rademacher function $w_S r_i \in {\set{B}}$, for any $i = 0, \ldots, n$.
Note that for any permutation $\sigma$ of $\{1,2,\ldots,q\}$,
\begin{equation*}
\one_{H^+_{i_1}} \one_{H^+_{i_2}} \cdots \one_{H^+_{i_q}} = \one_{H^+_{i_{\sigma(1)}}} \one_{H^+_{i_{\sigma(2)}}} \cdots \one_{H^+_{i_{\sigma(q)}}}.
\end{equation*} 
The functions $\one_{Q_k}$ and $\one_{H^+_{i_1}}\cdots \one_{H^+_{i_q}}$, $1 \leq i_1 \leq \ldots \leq i_q \leq k$, are linearly independent.
Indeed, assume that
\begin{equation*}
r:=\alpha_0 \one_{Q_k} + \sum_{i_1,\ldots,i_q \in \{1,\ldots,k\}}{\alpha_{i_{1}\cdots{i_{q}}}}\one_{H^+_{i_1}}\cdots\one_{H^+_{i_q}} = 0.
\end{equation*}
There are points $x_0\in \bigcap_{i=1}^k H_{i}^{-}$,
$x_{i_{1}}\ldots{x_{i_q}} \in \bigcap_{l\in\{i_1,\ldots,i_q\}} H_{l}^{-}
\cap
\bigcap_{l\neq {i_1,\ldots,i_q}}H_l^-$ for each $1 \leq i_1 \leq i_2 \leq \ldots \leq i_q \leq k$.
We obtain $\alpha_0=r(x_0)=0$ and ${\alpha_{i_{1}\cdots{i_{q}}}}=r(x_{i_{1}\cdots{i_{q}}})=0$ as needed.
\end{proof}

\subsection{Propagation of extrema, relative interior and the criterion of Barndorff-Nielsen}\label{s.pm}
Let $\set{B}$ be an arbitrary linear subspace of $\R^\set{X}$. 
Let $\set{B}'$ be the linear space spanned by $\set{B}$ and $\one$. 
Below we slightly generalize our condition on the existence of MLE for $e(\set{B})$.
\begin{lemma}\label{l.mpmp}
If $U\subset \set{X}$, then $\phi=\min_\set{X} \phi$ on $U$ implies $\phi=\min_\set{X} \phi$ on $\set{X}$ for every $\phi\in \set{B}$
if and only if $\phi=\max_\set{X} \phi$ on $U$ implies $\phi=\max_\set{X} \phi$ on $\set{X}$ for every $\phi\in \set{B}$.
\end{lemma}
\begin{proof}
The property with the minima is equivalent to the one with the maxima because $\set{B}$ is closed upon multiplication by $-1$ and because $\max (-\phi)=-\min \phi$.
\end{proof}
\begin{definition}
We say that $U\subset \set{X}$ propagates extrema for $\set{B}$ if $\phi=\inf_\set{X} \phi$ on $U$ implies that $\phi=\inf_\set{X} \phi$ on $\set{X}$ for every $\phi\in \set{B}$. 
\end{definition}
Due to Lemma~\ref{l.mpmp}, the property could be equivalently stated using maxima.
\begin{lemma}\label{l.m}
A non-empty $U\subset \set{X}$ propagates extrema for $\set{B}$ if and only if $U$ is of uniqueness for ${\set{B}}_+'$.
\end{lemma}
\begin{proof}
Assume that $U$ is of uniqueness for ${\set{B}}_+'$.
Let $\phi\in \set{B}$ and $\phi=\min_\set{X}\phi$ on $U$. Then $\varphi=\phi-\min_\set{X}\phi\in {\set{B}}_+'$ and $\varphi=0$ on $U$, so $\varphi=0$ on $\set{X}$ and $\phi=\min_\set{X} \phi$ on $\set{X}$. It follows that $U$ propagates extrema for $\set{B}$.
Conversely, assume that $U$ propagates extrema for $\set{B}$. Let $\phi\in {\set{B}}$. Then $\phi=\varphi+c$ for some $\varphi \in \set{B}$ and $c\in \R$. If $\phi\ge 0$ and $\phi=0$ on $U$, then $\varphi =\min_\set{X}\varphi=-c$ on $U$, hence $\varphi=-c$ on $\set{X}$, and so $\phi=0$ on $\set{X}$. Thus, $U$ is of uniqueness for ${\set{B}}_+'$.
\end{proof}
Theorem~\ref{thm:mle_exist} yields the following.
\begin{corollary}\label{c.pm}
MLE for $e(\set{B})$ and $x_1,...,x_n\in \set{X}$ exists iff $\{x_1,\ldots,x_n\}$ propagates extrema for $\set{B}$.
\end{corollary}
\begin{proof}
The MLE for $e(\set{B})$ and $e({\set{B}'})$ must be the same. Indeed,
we have $e(\set{B})=e({\set{B}'})$ so the suprema of the likelihood functions are the same, see Section~\ref{sec:exp_basic}.
Of course, if $\phi\in \set{B}$ and $e(\phi)$ is the MLE for $e(\set{B})$ then it is also the MLE for $e({\set{B}'})$.
Conversely, if $\phi\in {\set{B}'}$, 
then $\phi=\varphi+c$ for some $\varphi \in \set{B}$ and $c\in \R$. If $e(\phi)$ is the MLE for $e({\set{X}'})$, then $e(\varphi)$ is the 
MLE for $e(\set{B})$.
Considering ${\set{B}'}$, by Theorem~\ref{thm:mle_exist} we see that MLE for $e({\set{B}'})$ and $x_1,...,x_n\in \set{X}$ exists if and only if $\{x_1,\ldots,x_n\}$ is of uniqueness for ${\set{B}}_+'$, and -- by Lemma~\ref{l.m} -- if and only if $\{x_1,\ldots,x_n\}$ propagates extrema for $\set{B}$.
\end{proof}
Here is yet another formulation, which hinges on the trivial observation that if the sample mean equals the minimum, then the sample is constant.
\begin{lemma}\label{l.mbm}
$\{x_1,\ldots,x_n\}$ propagates extrema for $\set{B}$ if and only if for every $\phi\in \set{B}$, $\min_\set{X}\phi<\max_\set{X} \phi$ implies $\min_\set{X}\phi<\bar\phi< \max_\set{X} \phi$.
\end{lemma}
\begin{proof}
Let $\{x_1,\ldots,x_n\}$  propagate extrema for $\set{B}$. If $\min_\set{X}\phi=\bar\phi$, then $\phi=\min_\set{X}\phi$ on $\{x_1,\ldots,x_n\}$, hence $\phi=\min_\set{X}\phi$ on $\set{X}$ and so $\min_\set{X}\phi=\max_\set{X} \phi$. A similar argument works if $\bar\phi=\max_\set{X} \phi$; see also Lemma~\ref{l.mpmp}.
Conversely, if $\{x_1,\ldots,x_n\}$ does not propagate extrema for $\set{B}$ then there is $\phi\in \set{B}$ such that $\phi=\min_\set{X}\phi$ on $\{x_1,\ldots,x_n\}$, but $\max_\set{X} \phi>\min_\set{X}\phi$. Then $\min_\set{X}\phi=\bar\phi< \max_\set{X} \phi$.
\end{proof}
Recall the setting and notation of Section~\ref{ex:d}. The following theorem was essentially proved in \cite[Theorem~9.13]{MR489333}, except that it was stated for the minimal representation of exponential families. 
The formulation presented in Theorem~\ref{l.BN} below was given in \cite[Theorem~3.5]{MR558392}, which covers the arbitrary canonical representation and does so with a more direct proof. Notably, \cite{MR558392} uses the notion of relative interior of a convex set. 
Let $C$ be the convex hull of $S$. We say that $t\in \R^d$ is in the relative interior of $C$ if for every $\theta\in \Rd$, $\min_{y\in C} \theta\cdot y<\max_{y\in C} \theta\cdot y$ implies
$\min_{y\in C} \theta\cdot y<\theta\cdot t< \max_{y\in C} \theta\cdot y$. 
\begin{theorem}\label{l.BN}
MLE for $e(\set{B})$ and $x_1,...,x_n\in \set{X}$ exists, if and only if $\bar T$
is in the relative interior of $C$.
\end{theorem}
\begin{proof}
By the discussion in this section we know very well that MLE for $x_1,...,x_n$ and $e(\set{B})$ exists if and only if for every $\phi\in \set{B}$, $\min_\set{X}\phi<\max_\set{X} \phi$ implies $\min_\set{X}\phi<\bar\phi< \max_\set{X} \phi$.
Recall that $\phi\in \set{B}$ if and only if there is $\theta\in \R^d$ such that $\phi=\theta\cdot T$. Then
$\min_{x\in \set{X}}\phi(x)=\min_{y\in \set{S}}\theta\cdot y=\min_{y\in \set{C}}\theta\cdot y$, $\max_{x\in \set{X}}\phi(x)=\max_{y\in \set{C}}\theta\cdot y$, and, of course, $\bar\phi=\theta\cdot \bar T$.
Therefore the existence of MLE for $x_1,...,x_n$ and $e(\set{B})$ is equivalent to $\bar T$ being in the relative interior of $C$. 
\end{proof}
For clarity, we recall that we agreed in Example~\ref{ex:d} that the existence of MLE for $x_1,...,x_n\in \set{X}$ and $e(\set{B})$ is the same as the existence of MLE for $x_1,...,x_n$ and the exponential family given by the canonical statistics $T$ and \eqref{e.efcs}, and that it is equivalent to the existence of MLE for the sample $y_1:=T(x_1),...,y_n=T(x_n)\in \R^d$ and the standard exponential family in \eqref{e.sefT}. 
Also, we see from the above discussion that the convex hull $C$ and the notion of relative interior are merely auxiliary objects to express the property in Lemma~\ref{l.mbm}, or the propagation of extrema property.


\end{document}